\documentclass[11pt]{amsart}
\usepackage{latexsym,amscd,amssymb, graphicx, color, amsthm, bm, amsmath, cancel, enumitem}  
\usepackage{tikz}
\usepackage[margin=1in]{geometry}
\usepackage{hyperref}
\usepackage[all,cmtip]{xy}

\numberwithin{equation}{section}

\newtheorem{theorem}{Theorem}[section]
\newtheorem{proposition}[theorem]{Proposition}

\newtheorem{lemma}[theorem]{Lemma}
\newtheorem{conjecture}[theorem]{Conjecture}

\newtheorem{remark}[theorem]{Remark}

\theoremstyle{definition}
\newtheorem{defn}[theorem]{Definition}

\newcommand{\sign}{{\mathrm {sign}}}
\newcommand{\gr}{{\mathrm {gr}}}

\newcommand{\ann}{{\mathrm {ann}}}
\newcommand{\Stir}{{\mathrm {Stir}}}
\newcommand{\Hilb}{{\mathrm {Hilb}}}
\newcommand{\grFrob}{{\mathrm {grFrob}}}

\newcommand{\rev}{{\mathrm {rev}}}

\newcommand{\symm}{{\mathfrak{S}}}

\newcommand{\CC}{{\mathbb {C}}}

\newcommand{\RR}{{\mathbb{R}}}

\newcommand{\FF}{{\mathfrak{F}}}
\newcommand{\DD}{{\mathfrak{D}}}

\newcommand{\MMM}{{\mathcal{M}}}

\newcommand{\CCC}{{\mathcal{C}}}

\newcommand{\JJJ}{{\mathcal{J}}}
\newcommand{\HHH}{{\mathcal{H}}}
\newcommand{\AAA}{{\mathcal{A}}}
\newcommand{\BBB}{{\mathcal{B}}}

\newcommand{\RRRR}{{\mathfrak{R}}}
\newcommand{\III}{{\mathcal{I}}}

\newcommand{\NNNN}{{\mathfrak{N}}}

\newcommand{\ttheta}{{\bm \theta}}

\newcommand{\zz}{{\mathbf {z}}}
\newcommand{\xx}{{\mathbf {x}}}
\newcommand{\II}{{\mathbf {I}}}
\newcommand{\yy}{{\mathbf {y}}}

\newcommand{\aaa}{{\mathfrak {a}}}

\newcommand{\stair}{{\mathrm{st}}}

\newcommand{\VV}{{\mathbf {V}}}

\newcommand{\Gale}{{\mathrm {Gale}}}

\begin{document}

\title[The Hilbert series of the superspace coinvariant ring]
{The Hilbert series of the superspace coinvariant ring}

\author{Brendon Rhoades}
\address
{Department of Mathematics \newline \indent
University of California, San Diego \newline \indent
La Jolla, CA, 92093, USA}
\email{bprhoades@ucsd.edu}

\author{Andy Wilson}
\address
{Department of Mathematics \newline \indent
Kennesaw State University \newline \indent
Marietta, GA, 30060, USA}
\email{awils342@kennesaw.edu}

\begin{abstract}
Let $\Omega_n$ be the ring of polynomial-valued holomorphic differential forms on complex $n$-space, referred to in physics
as the superspace ring of rank $n$.
The symmetric group $\symm_n$ acts diagonally on $\Omega_n$ by permuting commuting and anticommuting generators simultaneously.
We let $SI_n \subseteq \Omega_n$ be the ideal generated by $\symm_n$-invariants with vanishing constant term and study the quotient
$SR_n = \Omega_n / SI_n$ of superspace by this ideal.
We calculate the doubly-graded Hilbert series of $SR_n$ and prove an `operator theorem' which characterizes the harmonic 
space $SH_n \subseteq \Omega_n$ attached to $SR_n$ in terms of the Vandermonde determinant and certain differential operators.
Our methods employ commutative algebra results which were used in the study of Hessenberg varieties.
Our results prove conjectures of N. Bergeron,  Li, Machacek, Sulzgruber, Swanson, Wallach, and Zabrocki.
\end{abstract}

\maketitle

\section{Introduction}
\label{Introduction}

Let $\xx_n = (x_1, \dots, x_n)$ be a list of $n$ variables and let $\CC[\xx_n]$ be the polynomial ring in these variables 
over $\CC$.
The symmetric group $\symm_n$ acts on $\CC[\xx_n]$ by subscript permutation; the fixed subspace
$\CC[\xx_n]^{\symm_n}$ is the algebra of {\em symmetric polynomials}. The {\em coinvariant ideal} $I_n \subseteq \CC[\xx_n]$
is the ideal $I_n := ( \CC[\xx_n]^{\symm_n}_+ )$ generated by the space $\CC[\xx_n]^{\symm_n}_+$ 
of symmetric polynomials with vanishing constant term and the {\em coinvariant ring}
$R_n := \CC[\xx_n] / I_n$
is the quotient of $\CC[\xx_n]$ by $I_n$.

The graded $\symm_n$-module $R_n$ is among the most important objects in algebraic combinatorics.
E. Artin proved \cite{Artin} that the `sub-staircase monomials' 
$\{ x_1^{a_1} \cdots x_n^{a_n} \,:\, a_i < i \}$ descend to a basis of $R_n$, so that 
$R_n$ has Hilbert series
\begin{equation}
\Hilb(R_n; q) = [n]!_q
\end{equation}
where we use the standard $q$-number and $q$-factorial notation
\begin{equation}
[n]_q := 1 + q + \cdots + q^{n-1} = \frac{1 - q^n}{1-q} \quad \quad \text{and} \quad \quad
[n]!_q := [n]_q [n-1]_q \cdots [1]_q. 
\end{equation}
Chevalley showed \cite{Chevalley} that $R_n \cong \CC[\symm_n]$ carries the regular representation of $\symm_n$
as an ungraded $\symm_n$-module and Borel showed \cite{Borel} that $R_n = H^{\bullet}(\mathrm{Fl}(n))$ presents the cohomology
of the type A complete flag variety.

Now let $\xx_n = (x_1, \dots, x_n)$ and $\yy_n = (y_1, \dots, y_n)$ be two sets of $n$ commuting variables and consider the polynomial
ring $\CC[\xx_n, \yy_n]$ over these variables with the diagonal action of $\symm_n$, viz.
\begin{equation}
w \cdot x_i := x_{w(i)} \quad \quad
w \cdot y_i := y_{w(i)} \quad \quad
(w \in \symm_n, \, \, 1 \leq i \leq n).
\end{equation}
Let $DI_n \subseteq \CC[\xx_n, \yy_n]$ be the ideal generated by the $\symm_n$-invariants with vanishing constant term.
Garsia and Haiman \cite{GH, HaimanQuotient} initiated the study of the {\em diagonal coinvariant ring}
\begin{equation}
DR_n := \CC[\xx_n, \yy_n] / DI_n.
\end{equation}
The quotient $DR_n$ is a doubly-graded $\symm_n$-module.
Haiman used the algebraic geometry of Hilbert schemes to prove \cite{Haiman} that $\dim DR_n = (n+1)^{n-1}$ and that,
as an ungraded $\symm_n$-module, the space $DR_n$ carries the sign-twisted permutation action of $\symm_n$ on size $n$ parking
functions. 
Carlsson and Oblomkov  used the Lusztig-Smelt paving of affine Springer fibers to give \cite{CO} a monomial basis 
of $DR_n$ which restricts to Artin's basis of $R_n$ when the $y$-variables are set to zero.

Next, let $\xx_n = (x_1, \dots, x_n)$ be a list of $n$ commuting variables and let $\ttheta_n = (\theta_1, \dots, \theta_n)$ be a list of $n$
anticommuting variables.
The {\em superspace} ring of rank $n$ is the tensor product
\begin{equation}
\Omega_n = \CC[\xx_n] \otimes \wedge \{ \ttheta_n \}
\end{equation}
of the polynomial ring in the $x$-variables and the exterior algebra over the $\theta$-variables.
This ring arises in physics, where the $x$-variables correspond to the states of bosons and the $\theta$-variables
correspond to the states of fermions; see e.g. \cite{PS}.
Accordingly, we shall refer to $x$-degree as {\em bosonic degree} and $\theta$-degree as {\em fermionic degree}.
The ring $\Omega_n$ also arises in differential geometry as the ring of polynomial-valued holomorphic differential forms on 
complex $n$-space (and we would write $dx_i$ instead of $\theta_i$); this explains our use of $\Omega$.

The symmetric group $\symm_n$ acts diagonally on superspace by the rule
\begin{equation}
w \cdot x_i = x_{w(i)} \quad \quad w \cdot \theta_i = \theta_{w(i)} \quad \quad (w \in \symm_n, \, \, 1 \leq i \leq n).
\end{equation}
Once again, we denote by $(\Omega_n)^{\symm_n}_+$ the subalgebra of invariant polynomials with vanishing constant term and 
consider the quotient ring
\begin{equation}
SR_n := \Omega_n / SI_n
\end{equation}
where the {\em supercoinvariant ideal} $SI_n \subseteq \Omega_n$ is given by
\begin{equation}
SI_n := \text{ideal generated by $(\Omega_n)^{\symm_n}_+$} \subseteq \Omega_n.
\end{equation}
Like $DR_n$, the quotient $SR_n$ is a bigraded $\symm_n$-module, this time with respect to bosonic and fermionic degree.

The study of $SR_n$ was initiated by the Fields Institute Combinatorics 
Group\footnote{Nantel Bergeron, Shu Xiao Li, John Machacek, 
Robin Sulzgruber, and Mike Zabrocki}
in roughly 2018.
This group conjectured that $\dim SR_n$ is the {\em ordered Bell number} counting ordered set partitions of $[n] := \{1, \dots, n \}$
and that, as an ungraded $\symm_n$-module, the quotient $SR_n$ carries the permutation action of $\symm_n$ on these 
ordered set partitions, up to sign twist.
Furthermore, this group conjectured that the doubly-graded $\symm_n$-structure of $SR_n$ was given by
\begin{equation}
\label{fields-conjecture}
\grFrob(SR_n; q, z) = \sum_{k = 1}^n z^{n-k} \cdot \Delta'_{e_{k-1}} e_n \mid_{t \rightarrow 0}
\end{equation}
where $q$ tracks bosonic degree, $z$ tracks fermionic degree, $e_n$ is the elementary symmetric function of degree $n$,
and $\Delta'_{e_{k-1}}$ is a {\em primed delta operator}
acting on the ring $\Lambda$ of symmetric functions; see \cite{HRW, Zabrocki} for more details. 
The identity \eqref{fields-conjecture} implies that the bigraded Hilbert series of $SR_n$ is given by
\begin{equation}
\label{fields-hilbert}
\Hilb(SR_n; q, z) = \sum_{k = 1}^n z^{n-k} \cdot [k]!_q \cdot \Stir_q(n,k)
\end{equation}
where the $q$-Stirling number $\Stir_q(n,k)$ is defined by the recursion
\begin{equation}
\Stir_q(n,k) = [k]_q \cdot \Stir_q(n-1,k) + \Stir_q(n-1,k-1)
\end{equation}
together with the initial condition
\begin{equation}
\Stir_q(0,k) = \begin{cases}
1 & k = 0 \\
0 & \text{otherwise}.
\end{cases}
\end{equation}
Equation~\eqref{fields-hilbert} was conjectured explicitly by Sagan and Swanson \cite[Conj. 6.5]{SS}.

The conjectures \eqref{fields-conjecture} and \eqref{fields-hilbert} were publicized at a BIRS meeting in January 2019.
This resulted in great excitement. Haglund, Rhoades, and Shimozono \cite{HRS} had introduced the quotient ring
\begin{equation}
R_{n,k} := \CC[\xx_n] / (x_1^k, x_2^k, \dots, x_n^k, e_n, e_{n-1}, \dots, e_{n-k+1})
\end{equation}
and had proven \cite{HRS2} that 
\begin{equation}
\grFrob(R_{n,k};q) = (\rev_q \circ \omega) \Delta'_{e_{k-1}} e_n \mid_{t = 0}.
\end{equation}
Pawlowski and Rhoades \cite{PR} introduced the moduli space $X_{n,k}$ of $n$-tuples of lines $(\ell_1, \dots, \ell_n)$
in $\CC^k$ such that $\ell_1 + \cdots + \ell_k = \CC^k$ and proved the cohomology presentation
\begin{equation}
H^{\bullet}(X_{n,k}) = R_{n,k}.
\end{equation}
The authors \cite{RWVan}  introduced the {\em superspace Vandermonde}
\begin{equation}
\delta_{n,k} := \varepsilon_n \cdot \left(
x_1^{k-1} \cdots x_{n-k}^{k-1} x_{n-k+1}^{k-1} x_{n-k+2}^{k-2} \cdots x_{n-1}^1 x_n^0 \times \theta_1 \cdots \theta_{n-k}
\right)
\end{equation}
and showed that the subspace $V_{n,k} \subseteq \Omega_n$ obtained by starting with $\delta_{n,k}$ and closing under 
the partial derivative operators $\frac{\partial}{\partial x_i}$ and linearity carries a graded $\symm_n$-action with 
graded character $\Delta'_{e_{k-1}} e_n \mid_{t = 0}$.
Of all of these models, the supercoinvariant ring $SR_n$ has the most intrinsic invariant-theoretic definition which extends to
arbitrary complex reflection groups $G \subseteq GL_n(\CC)$ in the most obvious way.

Zabrocki extended the conjecture \eqref{fields-conjecture} in a different direction
by introducing another set of commuting variables $\yy_n = (y_1, \dots, y_n)$ 
 and considering the triply-graded $\symm_n$-module obtained by quotienting
$\CC[\xx_n, \yy_n] \otimes \wedge \{ \ttheta_n \}$ by the ideal $I$ generated by  $\symm_n$-invariants with vanishing constant term.
Zabrocki conjectured \cite{Zabrocki} that 
\begin{equation}
\label{zabrocki-conjecture}
\grFrob \left(
\CC[\xx_n, \yy_n] \otimes \wedge \{ \ttheta_n \} / I ; q, t, z 
\right) = \sum_{k = 1}^n z^{n-k} \cdot \Delta'_{e_{k-1}} e_n
\end{equation}
where $q$ tracks $x$-degree, $t$ tracks $y$-degree, and $z$ tracks $\theta$-degree.
Observe that \eqref{zabrocki-conjecture} reduces to \eqref{fields-conjecture} if the $y$-variables are set to zero, and 
Haiman's theorem \cite{Haiman} when the $\theta$-variables are set to zero.
The conjecture \eqref{zabrocki-conjecture} was the first predicted algebraic model for $\Delta'_{e_{k-1}} e_n$;
the authors \cite{RWVan} gave a parallel conjectural model for $\Delta'_{e_{k-1}} e_n$ involving the superspace Vandermondes 
$\delta_{n,k}$.
The conjecture \eqref{zabrocki-conjecture} was extended to two sets of bosonic variables and two sets of fermionic variables by 
D'Adderio, Iraci, and Vanden Wyngaerd \cite{DIV} using $\Theta$-operators on symmetric functions;
the case of two sets of fermionic variables alone was solved by Iraci-Rhoades-Romero \cite{IRR} and Kim-Rhoades \cite{KR1};
see \cite{Kim, KR2} for a connection between this quotient and skein relations on set partitions.
F. Bergeron has a substantial family \cite{BergeronMulti, BergeronMultiSkew, Bergeron} of conjectures on coinvariant quotients
with multiple sets of bosonic and fermionic variables. 

Despite all of this activity, the equations \eqref{fields-conjecture} and \eqref{fields-hilbert} on the structure of $SR_n$
remained frustratingly conjectural.
The methods which were used to successfully analyze objects like $R_{n,k}, X_{n,k},$ and $V_{n,k}$ have not yet been extended
to study $SR_n$.
Swanson and Wallach \cite{SW1, SW2} proved that the $\sign$-isotypic component of \eqref{fields-conjecture} is correct,
 and that the fermionic degree $n-k$ piece of $SR_n$ has top bosonic degree $(n-k) \cdot (k-1) + {k \choose 2}$ 
 as predicted by \eqref{fields-hilbert}; this was the only significant progress on $SR_n$.
 In fact, before this paper, even the dimension of $SR_n$ was unknown.

In this paper we will prove that the formula \eqref{fields-hilbert} calculates the bigraded Hilbert series of $SR_n$
(Theorem~\ref{hilbert-series}). We will also prove 
(Theorem~\ref{superharmonic-space-characterization})
an `operator conjecture' of 
Swanson and Wallach \cite{SW2} which describes the harmonic space $SH_n \subseteq \Omega_n$ attached to the 
supercoinvariant ring $SR_n$ using certain `higher Euler operators' on $\Omega_n$  which act by polarization.\footnote{This 
characterization of $SH_n$ was conjectured earlier in unpublished work of N. Bergeron, S. X. Li, J. Machacek,
R. Sulzgruber, and M. Zabrocki.}
The space $SH_n$ is helpful for machine computations because $SH_n \cong SR_n$ as doubly-graded $\symm_n$-modules,
and yet members of $SH_n$ are honest superspace elements $f \in \Omega_n$ rather than cosets $f + SI_n \in SR_n$.
The $\symm_n$-module structure of $SR_n$, ungraded or (bi)graded, remains open.

We turn to a description of our methods.
The analysis of $R_{n,k}$ and its variations relied on the remarkably well-behaved Gr\"obner theory of its defining ideal
$(x_1^k, \dots, x_n^k, e_n, \dots, e_{n-k+1}) \subseteq \CC[\xx_n]$.
This facilitated multiple provable combinatorial bases \cite{GR, HRS, Meyer, PR} of $R_{n,k}$ from which its structure 
as a graded vector space or $\symm_n$-module could be studied.
There exists an extension of Gr\"obner theory to the superspace ring $\Omega_n$, but the Gr\"obner theory of the supercoinvariant
ideal $SI_n \subseteq \Omega_n$ has proven to be inscrutable.
Combinatorially, this has translated into a failure of using straightening arguments to
show that nice potential bases of $SR_n$ span this quotient ring.
Indeed, our approach does not prove the existence of any specific basis of $SR_n$. For a potential road from our methods
to an Artin-like basis of $SR_n$ conjectured by Sagan and Swanson \cite[Conj. 6.7]{SS}, see
Theorem~\ref{basis-recipe}, Conjecture~\ref{artin-conjecture}, and Proposition~\ref{artin-special-case}.

Since the direct analysis of $SR_n$ by means of a basis has proven elusive, we adopt an indirect approach which stands,
in a nutshell, on the elimination of fermionic variables. This allows us to trade supercommutative algebra problems in 
$\Omega_n$ for commutative algebra problems in $\CC[\xx_n]$, for which more tools have been developed.

For a given subset $J \subseteq [n]$,
we use a miraculous identity (Lemma~\ref{miracle-identity}) involving partial derivatives of complete homogeneous symmetric 
polynomials to deduce the existence of
a regular sequence $p_{J,1}, \dots, p_{J,n} \in \CC[\xx_n]$
 (Lemma~\ref{regular-sequence-lemma}) in $\CC[\xx_n]$.
 These regular sequences are used to prove (Proposition~\ref{upper-bound-dimension}) that the bigraded Hilbert series
 of $SR_n$ is bounded above by the expression \eqref{fields-hilbert}.

Next, we introduce a family $\DD_J$ of combinatorially defined
differential operators acting on $\Omega_n$ which are indexed by subsets $J \subseteq [n]$.
We prove (Lemma~\ref{f-product}) that the $\DD_J$ exhibit a triangularity property with respect to the Gale order on subsets 
$J \subseteq [n]$ with leading term given by the polynomial
\begin{equation}
f_J := \prod_{j \in J} x_j \left(
\prod_{i \, > \, j} (x_j - x_i)
\right) \in \CC[\xx_n].
\end{equation}
This leads to a general recipe (Theorem~\ref{basis-recipe}) for constructing bases of $SR_n$ from bases of the various
commutative quotient rings 
$\CC[\xx_n] / (I_n : f_J)$ by the colon ideals
\begin{equation}
(I_n : f_J) := \{ g \in \CC[\xx_n] \,:\, g \cdot f_J \in I_n \}.
\end{equation}
By identifying $(I_n : f_J)$ with the ideal $(p_{J,1}, \dots, p_{J,n})$ cut out by the regular sequence in $\CC[\xx_n]$ 
used to prove the upper bound on $\Hilb(SR_n;q,z)$ (Theorem~\ref{colon-ideal-identification}),
we are able to prove both the operator theorem characterizing the superharmonic space $SH_n$
(Theorem~\ref{superharmonic-space-characterization})
and the formula \eqref{fields-hilbert} for the bigraded Hilbert series of $SR_n$
(Theorem~\ref{hilbert-series}).

The rest of the paper is organized as follows. 
In {\bf Section~\ref{Background}} we give background material on superspace and commutative algebra.
In {\bf Section~\ref{Upper}} we bound the bigraded Hilbert series of $SR_n$ from above using regular sequences.
In {\bf Section~\ref{Differential}} we introduce the differential operators $\DD_J$ and relate them to the colon ideals $(I_n : f_J)$.
In {\bf Section~\ref{Operator}} we prove our main results: the operator theorem and the Hilbert series of $SR_n$.
We also present a conjecture for an Artin-like basis of $\CC[\xx_n] / (I_n : f_J)$ and prove this conjecture in a special case.
We close in {\bf Section~\ref{Conclusion}} with some open problems.

\section{Background}
\label{Background}

\subsection{Superspace}
As  in the introduction, the superspace ring 
$\Omega_n = \CC[\xx_n] \otimes \wedge \{ \ttheta_n \}$ is the tensor product 
of a symmetric algebra of rank $n$ and an exterior algebra of rank $n$,
both over  $\CC$.
A {\em monomial} in $\Omega_n$ is a nonzero product of the generators $\xx_n = (x_1, \dots, x_n)$ 
and $\ttheta_n = (\theta_1, \dots, \theta_n)$.
A {\em bosonic monomial} is a monomial which only involves the generators $\xx_n$ whereas a {\em fermionic monomial}
is a monomial which only involves the generators $\ttheta_n$.
For any subset $J \subseteq [n]$, we let $\theta_J$ be the product of the fermionic generators $\theta_j$
indexed by $j \in J$ in increasing order; we have a direct sum decomposition
\begin{equation}
\Omega_n = \bigoplus_{J \subseteq [n]} \CC[\xx_n] \cdot \theta_J.
\end{equation}

The {\em Gale order} $\leq_\Gale$ on subsets $J \subseteq [n]$ of the same cardinality will be used heavily. This partial order is defined 
by
\begin{equation}
\{ a_1 < \cdots < a_r \} \leq_\Gale \{ b_1 < \cdots < b_r \} \text{ if $a_i \leq b_i$ for all $i$.}
\end{equation}
This order will be used to compare fermionic monomials $\theta_J$ in the superspace ring $\Omega_n$.

The ring $\Omega_n$ may be identified with polynomial valued differential forms on $\CC^n$; as such, it carries
a plethora of derivative operators.  For $1 \leq i \leq n$, let $\partial_i: \CC[\xx_n] \rightarrow \CC[\xx_n]$ be the usual
partial differentiation with respect to $x_i$.
By acting on the first tensor factor of $\Omega_n = \CC[\xx_n] \otimes \wedge \{ \ttheta_n \}$,
 this extends to an action $\partial_i: \Omega_n \rightarrow \Omega_n$.
For $1 \leq i \leq n$, let $\partial^{\theta}_i: \wedge \{ \ttheta_n \} \rightarrow \wedge \{ \ttheta_n \}$ be the {\em contraction}
operator defined on fermionic monomials by
\begin{equation}
\partial^\theta_i: \theta_{j_1} \cdots \theta_{j_r} = \begin{cases}
(-1)^{s-1} \theta_{j_1} \cdots \widehat{\theta_{j_s}} \cdots \theta_{j_r} & \text{if $j_s = i$ for some $s$,} \\
0 & \text{otherwise}
\end{cases}
\end{equation}
for any distinct indices $1 \leq j_1, \dots, j_r \leq n$ where $\widehat{\cdot}$ denotes omission.
By acting on the second tensor factor of $\Omega_n = \CC[\xx_n] \otimes \wedge \{ \ttheta_n \}$,
we have a fermionic derivative operator $\partial^{\theta}_i: \Omega_n \rightarrow \Omega_n$.

We let $d: \Omega_n \rightarrow \Omega_n$ be the Euler operator of differential geometry defined by
\begin{equation}
d: f \mapsto \sum_{i = 1}^n \partial_i f \cdot \theta_i
\end{equation}
for all $f \in \Omega_n$. This operator lowers bosonic degree by 1 while raising fermionic degree by 1.
We will need `higher' versions $d_j: \Omega_n \rightarrow \Omega_n$
 $(j \geq 1)$ of these operators given by
\begin{equation}
d_j: f \mapsto \sum_{i = 1}^n \partial_i^j f \cdot \theta_i.
\end{equation}
The operator $d_j$ decreases bosonic degree by $j$ while raising fermionic degree by 1. We have 
$d_1 = d$. If $J = \{ j_1 < j_2 < \cdots \}$
is a set of positive integers, we write 
\begin{equation}
d_J := d_{j_1} d_{j_2} \cdots
\end{equation}
for the corresponding product of higher Euler operators.

Considering bosonic and fermionic degree separately, superspace $\Omega_n$ admits a bigrading
\begin{equation}
\Omega_n = \bigoplus_{i \geq 0} \bigoplus_{j = 0}^n (\Omega_n)_{i,j} \quad \text{ where } \quad 
(\Omega_n)_{i,j} = \CC[\xx_n]_i \otimes \wedge^j \{ \ttheta_n \}.
\end{equation}
The diagonal action of the symmetric group $\symm_n$ on $\Omega_n$ preserves this bigrading.
As in the introduction, we let $(\Omega_n)^{\symm_n}$ be the fixed subalgebra for this action.

Let $I \subseteq \Omega_n$ be a bihomogeneous ideal in superspace (such as $SI_n$). Analysis of the quotient ring
$\Omega_n/I$ is often complicated by the fact that its elements $f + I$ are cosets rather than  superspace elements $f \in \Omega_n$.
The theory of {\em (superspace) harmonics} is a powerful technique for replacing cosets with honest elements of superspace.
We turn to a description of this method.

The partial derivative operators $\partial_i, \partial^\theta_i: \Omega_n \rightarrow \Omega_n$ satisfy the relations
\begin{equation}
\partial_i \partial_j = \partial_j \partial_i \quad \quad
\partial_i \partial^{\theta}_j = \partial^{\theta}_j \partial_i \quad \quad
\partial^{\theta}_i \partial^{\theta}_j = - \partial^{\theta}_j \partial^{\theta}_i
\end{equation}
for all $1 \leq i, j \leq n$.  Since these are the defining relations of $\Omega_n$, for any superspace element 
$f = f(x_1, \dots, x_n, \theta_1, \dots, \theta_n) \in \Omega_n$ we get an operator 
\begin{equation}
\partial f = f(\partial_1, \dots, \partial_n, \partial^{\theta}_1, \dots, \partial^{\theta}_n): \Omega_n \rightarrow \Omega_n
\end{equation}
by replacing each $x_i$ in $f$ with the bosonic derivative $\partial_i$ and each $\theta_i$ in $f$ with the fermionic 
derivative $\partial^{\theta}_i$. This leads to an action of superspace on itself given by
\begin{equation}
\odot: \Omega_n \times \Omega_n \rightarrow \Omega_n \quad \quad
f \odot g := (\partial f)(g).
\end{equation}
The $\odot$-action gives $\Omega_n$-module structure on $\Omega_n$.

We use the $\odot$-action to construct an inner product on $\Omega_n$ as follows. 
Let $\overline{\cdot}: \Omega_n \rightarrow \Omega_n$ be the conjugate-linear involution which fixes all bosonic monomials,
satisfies $\overline{ \theta_{i_1} \cdots \theta_{i_r} } = \theta_{i_r} \cdots \theta_{i_1}$ for all fermionic monomials
$\theta_{i_1} \cdots \theta_{i_r}$, and sends any scalar $c \in \CC$ to its complex conjugate $\overline{c}$.
The pairing
\begin{equation}
\langle -, - \rangle: \Omega_n \times \Omega_n \rightarrow \Omega_n \quad \quad 
\langle f, g \rangle := \text{constant term of $f \odot \overline{g}$}. 
\end{equation}
is easily seen to be an inner product, with the monomials $\{ x_1^{a_1} \cdots x_n^{a_n} \cdot \theta_I \}$ forming an
orthogonal (but not orthonormal) basis.

Now suppose $I \subseteq \Omega_n$ is a bihomogeneous ideal. We have the equality 
\begin{equation}
I^{\perp} = \{ g \in \Omega_n \,:\, f \odot g = 0 \text{ for all $f \in I$} \}
\end{equation}
of subspaces of $\Omega_n$, where $I^{\perp}$ is calculated with respect to the above inner product.
The subspace $I^{\perp} \subseteq \Omega_n$ is the {\em harmonic space} attached to $I$.
We have a direct sum decomposition $\Omega_n = I \oplus I^{\perp}$ and an isomorphism of bigraded vector spaces
$\Omega_n/I \cong I^{\perp}$.  If $I$ is $\symm_n$-stable, the isomorphism 
$\Omega_n/I \cong I^{\perp}$ is also an isomorphism of bigraded $\symm_n$-modules.
The harmonic model $I^{\perp}$ of $\Omega_n/I$ is  useful because its members are honest superspace elements rather than cosets.

We close this subsection with a combinatorial identity due to Sagan and Swanson which will be useful in our analysis
of $SR_n$.  For a subset $J \subseteq [n]$, we define the {\em $J$-staircase} to be the sequence
$\stair(J) = (\stair(J)_1, \dots, \stair(J)_n)$
where
\begin{equation}
\stair(J)_1 := \begin{cases}
0 & 1 \in J \\
1 & 1 \notin J
\end{cases}
\end{equation}
and 
\begin{equation}
\stair(J)_{i+1} := \begin{cases}
\stair(J)_i & i+1 \in J \\
\stair(J)_i + 1 & i+1 \notin J.
\end{cases}
\end{equation}
For example, if $n = 7$ and $J = \{ 3,5,6\}$ we have $\stair(J) = (\stair(J)_1, \dots, \stair(J)_7) = (1,2,2,3,3,3,4)$.
Observe that $\stair(\varnothing) = (1, 2, \dots, n)$ is the usual staircase.

\begin{lemma}
\label{ss-dimension}
{\em (Sagan-Swanson \cite{SS})}
We have the polynomial identity
\begin{equation}
\sum_{J \subseteq [n]} \left(  \prod_{i = 1}^n [\stair(J)_i]_q  \right) \cdot z^{|J|} = 
\sum_{k = 1}^n z^{n-k} \cdot [k]!_q \cdot \Stir_q(n,k).
\end{equation}
\end{lemma}

\subsection{Commutative Algebra}
 Our overarching strategy for analyzing $SR_n$ is to transfer problems involving
the superspace ring $\Omega_n$ to problems involving the better-understood polynomial ring $\CC[\xx_n]$.
We review the relevant notions from commutative algebra.

A commutative graded $\CC$-algebra 
$A = \bigoplus_{i \geq 0} A_i$ is {\em Artinian} if 
$A$ is a finite-dimensional $\CC$-vector space.
The {\em Hilbert series} of $A$ is 
\begin{equation}
\Hilb(A;q) := \sum_{i \geq 0} \dim_\CC(A_i) \cdot q^i,
\end{equation}
assuming each graded piece $A_i$ is finite-dimensional.

A sequence $f_1, \dots, f_n$ of $n$ polynomials in $\CC[\xx_n]$ of homogeneous positive degrees 
is a {\em regular sequence} if, for each $0 \leq i \leq n-1$,
we have a short exact sequence
\begin{equation}
0 \rightarrow 
\CC[\xx_n]/(f_1, \dots, f_i) \xrightarrow{\, \, \times f_{i+1} \, \, }
\CC[\xx_n]/(f_1, \dots, f_i) \xrightarrow{ \, \, \mathrm{can.} \, \,}
\CC[\xx_n]/(f_1, \dots ,f_i, f_{i+1}) \rightarrow 0
\end{equation}
where the first map is induced by multiplication by $f_{i+1}$ and the second map is the canonical projection.
If the regular sequence $f_1, \dots, f_n$ consists of homogeneous polynomials, the quotient ring 
$\CC[\xx_n]/(f_1, \dots, f_n)$ is a finite-dimensional graded vector space with Hilbert series
\begin{equation}
\Hilb(\CC[\xx_n]/(f_1, \dots, f_n); q) = 
[\deg f_1]_q \cdots [\deg f_n]_q.
\end{equation}
An Artinian graded quotient $\CC[\xx_n]/\aaa$ of $\CC[\xx_n]$ is a {\em complete intersection} if 
$\aaa = (f_1, \dots, f_n)$ for some length $n$ regular sequence $f_1, \dots, f_n \in \CC[\xx_n]$ of homogeneous polynomials.

The regularity of a sequence $f_1, \dots, f_n \in \CC[\xx_n]$ of polynomials of homogeneous positive degree
can be interpreted in terms of the variety cut out by $f_1, \dots, f_n$.  Given any set $S \subseteq \CC[\xx_n]$
of polynomials, write
\begin{equation}
\VV(S) := \{ \zz \in \CC^n \,:\, f(\zz) = 0 \text{ for all $f \in S$} \}
\end{equation}
for the locus of points in $\CC^n$ on which the polynomials in $S$ vanish.

\begin{lemma}
\label{regular-sequence-criterion}
Let $f_1, \dots, f_n \in \CC[\xx_n]$ be a list of $n$ homogeneous polynomials in $\CC[\xx_n]$ of positive degree.
The sequence $f_1, \dots, f_n$ is a regular sequence if and only if the variety 
$\VV(f_1, \dots, f_n) \subseteq \CC^n$ cut out by these polynomials consists of the origin $\{ 0 \}$ alone.
\end{lemma}

Let $\aaa \subseteq \CC[\xx_n]$ be an ideal and let $f \in \CC[\xx_n]$ be a polynomial. The {\em colon ideal}
(or {\em ideal quotient}) is 
\begin{equation}
(\aaa : f) := \{ g \in \CC[\xx_n] \,:\, f \cdot g \in \aaa \} \subseteq \CC[\xx_n].
\end{equation}
It is not difficult to check that $(\aaa : f)$ is an ideal in $\CC[\xx_n]$ which contains $\aaa$, and that
$(\aaa : f) = \CC[\xx_n]$ if and only if $f \in \aaa$.

Colon ideals will play a crucial role in our work, and we will need a criterion for determining a generating set for them. 
Let $A = \bigoplus_{i = 0}^d A_i$ be a finite-dimensional graded $\CC$-algebra with $A_d \neq 0$.
The algebra $A$ is a {\em Poincar\'e duality algebra} if 
\begin{itemize}
\item its top component $A_d \cong \CC$ is a 1-dimensional complex vector spaces, and
\item for any $0 \leq i \leq d$, the multiplication map $A_i \otimes A_{d-i} \longrightarrow A_d \cong \CC$ is
a perfect pairing.
\end{itemize}
If $A = \bigoplus_{i = 0}^d A_d$ is a Poincar\'e duality algebra with $d \neq 0$, the maximal degree $d$ is called the 
{\em socle degree} of $A$.
The following commutative algebra lemma will be remarkably useful to us.

\begin{lemma}
\label{colon-ideal-equality}
{\em (Abe-Horiguchi-Masuda-Murai-Sato \cite[Lem. 2.4]{AHMMS})}
Suppose $\aaa, \aaa' \subseteq \CC[\xx_n]$ are homogeneous ideals and $f \in \CC[\xx_n]$ is a homogeneous polynomial
of degree $k$
with $f \notin \aaa$.  Suppose $\aaa' \subseteq (\aaa : f)$. If $\CC[\xx_n]/\aaa'$ is a Poincar\'e duality algebra of socle degree $r$
and $\CC[\xx_n]/\aaa$ is a Poincar\'e duality algebra of socle degree $r + k$, then  $\aaa' = (\aaa : f)$.
\end{lemma}

We remark that \cite[Lem. 2.4]{AHMMS} was stated over the field $\RR$ of real numbers, but its proof goes through
without change for arbitrary fields.

The polynomial ring $\CC[\xx_n]$ inherits a theory of harmonics from the superspace ring $\Omega_n$.
Partial differentiation yields an action $\odot: \CC[\xx_n] \times \CC[\xx_n] \rightarrow \CC[\xx_n]$ of 
the polynomial ring $\CC[\xx_n]$ on itself which gives rise to an inner product
\begin{equation}
\langle -, - \rangle: \CC[\xx_n] \times \CC[\xx_n] \rightarrow \CC \quad \quad 
\langle f, g \rangle = \text{ constant term of } f \odot \overline{g}.
\end{equation}
If $I \subseteq \CC[\xx_n]$ is a homogeneous ideal, we have a direct sum decomposition $\CC[\xx_n] = I \oplus I^{\perp}$
and an identification
\begin{equation}
I^{\perp} = \{ g \in \CC[\xx_n] \,:\, f \odot g = 0 \text{ for all $f \in I$} \}
\end{equation}
of the harmonic space $I^{\perp}$ as a subspace of $\CC[\xx_n]$.

The harmonic theory of the classical coinvariant ideal $I_n \subseteq \CC[\xx_n]$ is given as follows. Let 
$\delta_n \in \CC[\xx_n]$ be the {\em Vandermonde determinant}
\begin{equation}
\delta_n := \prod_{i < j} (x_j - x_i) \in \CC[\xx_n].
\end{equation}
Then $I_n^{\perp}$ is a cyclic $\CC[\xx_n]$-module under the $\odot$-action generated by $\delta_n$.
In symbols, we have
\begin{equation}
I_n^{\perp} = \CC[\xx_n] \odot \delta_n.
\end{equation}
We write $H_n$ for the subspace $I_n^{\perp} = \CC[\xx_n] \odot \delta_n \subseteq \CC[\xx_n]$; we have 
an isomorphism $R_n \cong H_n$ of graded $\symm_n$-modules.
The annihilator of $\delta_n$ under the $\odot$-action is precisely the coinvariant ideal $I_n$:
\begin{equation}
\ann_{\CC[\xx_n]}(\delta_n) = 
\{ f \in \CC[\xx_n] \,:\, f \odot \delta_n = 0 \} = I_n.
\end{equation}

\section{Upper Bound}
\label{Upper}

\subsection{A regular sequence in $\CC[\xx_n]$}
Our first lemma gives a general technique for constructing interesting elements of the supercoinvariant ideal $SI_n$.

\begin{lemma}
\label{si-differential-ideal}
The supercoinvariant ideal $SI_n \subseteq \Omega_n$
contains the classical coinvariant ideal $I_n \subseteq \CC[\xx_n]$ and
 is closed under the action of the Euler operator
$d: \Omega_n \rightarrow \Omega_n$.
\end{lemma}

\begin{proof}
The operator $d$ commutes with the action of $\symm_n$ on $\Omega_n$, so
the result follows from the Leibniz formula 
\begin{equation}
d(fg) = df \cdot g \pm f \cdot dg \quad \quad 
\end{equation} 
which holds for any bihomogeneous $f, g \in \Omega_n$ 
(the sign is $+$ if $f$ has even fermionic degree and $-$ otherwise)
and the relation $d \circ d = 0$.
\end{proof}

Ideals in $\Omega_n$ which are closed under the action of $d$ are called {\em differential ideals}.
To the knowledge of the authors, the supercoinvariant ideal $SI_n$ is the first differential ideal which has received 
significant attention in algebraic combinatorics.

The most important elements of $SI_n$ arising from Lemma~\ref{si-differential-ideal} are as follows. Let $h_r, e_r \in \CC[\xx_n]$
be the complete homogeneous and elementary symmetric polynomials
\begin{equation}
h_r := \sum_{1 \leq i_1 \leq \cdots \leq i_r \leq n} x_{i_1} \cdots x_{i_r} \quad \quad
e_r := \sum_{1 \leq i_1 < \cdots < i_r \leq n} x_{i_1} \cdots x_{i_r}.
\end{equation}
Here and throughout, if $S \subseteq [n]$ is an index set, we use $h_r(S)$ and $e_r(S)$ to denote the complete 
homogeneous and elementary symmetric polynomials of degree $r$ in the variables indexed by $S$. For example, we have
\begin{equation*}
h_2(134) = x_1^2 + x_1 x_3 + x_1 x_4 + x_3^2 + x_3 x_4 + x_4^2 \quad  \text{and} \quad
e_2(134) = x_1 x_3 + x_1 x_4 + x_3 x_4.
\end{equation*}
For any subset $S \subseteq [n]$,
it is well-known that 
\begin{equation}
\label{classical-coinvariant-ideal-membership}
h_r(S) \in I_n  \quad \text{whenever $r > n - |S|$.}
\end{equation}
Indeed, \eqref{classical-coinvariant-ideal-membership} follows inductively from the 
identity $h_r(S \cup i) = x_i h_{r-1}(S \cup i) + h_r(S)$ which holds whenever $i \notin S$. By Lemma~\ref{si-differential-ideal},
we have
\begin{equation}
\label{supercoinvariant-ideal-membership}
dh_r(S) \in SI_n  \quad \text{whenever $r > n - |S|$.}
\end{equation}
Elements of $SI_n$ of the form \eqref{classical-coinvariant-ideal-membership} and \eqref{supercoinvariant-ideal-membership}
are the only ones we will need.

For any subset $J \subseteq [n]$, we construct a sequence $(q_{J,1}, q_{J,2}, \dots, q_{J,n})$  of superspace elements
as follows. 
Given $J \subseteq [n]$, the sequence $(q_{J,1}, q_{J,2}, \dots, q_{J,n})$ in $\Omega_n$ is defined by
\begin{equation}
q_{J,i} := \begin{cases}
h_i(\{i, i+1, \dots, n\}) \cdot \theta_J & i < \min(J) \\
dh_r(J \cup \{i+1, \dots, n\}) \cdot \theta_{J - \max(J \cap \{1, \dots, i\})} & i \geq \min(J)
\end{cases}
\end{equation}
where in the second branch $r = n - |J \cup \{i+1, \dots, n\}| + 1$.

The superspace elements $q_{J,i}$ may be visualized (and remembered)
as follows. Consider a linear array of $n$ boxes labeled $1, \dots, n$ from left to right, where the boxes in positions $j \in J$
are decorated with a $\theta$. We consider moving a pointer from left to right along this array.
When $n = 7$ and $J = \{3,5,6\}$, the picture is shown in Figure~\ref{pointer-figure}.
\begin{itemize}
\item
When the pointer is at a position $i$ which is strictly to the left of all of the $\theta$ decorations, the corresponding superspace
element is
$q_{J,i} = h_i(\{i, i+1, \dots, n\}) \cdot \theta_J$.  
\item
When the pointer is at a position $i$ which is weakly to the right of at least
one $\theta$ decoration, the corresponding superspace element is 
$q_{J,i} = dh_r(J \cup \{i+1, \dots, n \}) \cdot \theta_{\overline{J}}$, where $\overline{J}$ consists of all elements of $J$ except for the 
closest element $j \in J$ weakly to the right of the pointer and $r = n - |J \cup \{i+1, \dots, n\}| + 1$ is the minimal degree such that 
$h_r(J \cup \{i+1, \dots, n\}) \in I_n$ lies in the classical coinvariant ideal.
\end{itemize}
In our example, we have
\begin{equation*}
q_{J,1} = h_1(1234567) \cdot \theta_{356} \quad 
q_{J,2} = h_2(234567) \cdot \theta_{356} \quad 
q_{J,3} = d h_3(34567) \cdot \theta_{56} \quad 
q_{J,4} = d h_4(3567) \cdot \theta_{56} 
\end{equation*}
\begin{equation*}
q_{J,5} = d h_4(3567) \cdot \theta_{36} \quad
q_{J,6} = d h_4(3567) \cdot \theta_{35} \quad
q_{J,7} = d h_5(356) \cdot \theta_{35}.
\end{equation*}
We record some basic observations about the polynomials $q_{J,i}$.

\begin{figure}
\begin{center}
\begin{tikzpicture}[scale = 0.4]
  \begin{scope}
    \clip (0,0) -| (7,1) -| (0,0);
    \fill [color=black!10] (0,0) -| (7,1) -| (0,0);
    \draw [color=black!70] (0,0) grid (7,1);
\end{scope}

\node at (0.5,1.9) {$\downarrow$};

\node at (2.5, 0.5) {$\theta$};
\node at (4.5, 0.5) {$\theta$};
\node at (5.5, 0.5) {$\theta$};

\node at (0.5,-0.7) {1};
\node at (1.5,-0.7) {2};
\node at (2.5,-0.7) {3};
\node at (3.5,-0.7) {4};
\node at (4.5,-0.7) {5};
\node at (5.5,-0.7) {6};
\node at (6.5,-0.7) {7};

\draw [thick] (0,0) -| (7,1) -| (0,0);
\draw [very thick, color = red] (2,-1.5) -- (2,3);

  \begin{scope}
    \clip (10,0) -| (17,1) -| (10,0);
    \fill [color=black!10] (11,0) -| (17,1) -| (11,0);
    \draw [color=black!70] (10,0) grid (17,1);
\end{scope}

\draw [thick] (10,0) -| (17,1) -| (10,0);
\draw [very thick, color = red] (12,-1.5) -- (12,3);

\node at (11.5,1.9) {$\downarrow$};

\node at (12.5, 0.5) {$\theta$};
\node at (14.5, 0.5) {$\theta$};
\node at (15.5, 0.5) {$\theta$};

\node at (10.5,-0.7) {1};
\node at (11.5,-0.7) {2};
\node at (12.5,-0.7) {3};
\node at (13.5,-0.7) {4};
\node at (14.5,-0.7) {5};
\node at (15.5,-0.7) {6};
\node at (16.5,-0.7) {7};

  \begin{scope}
    \clip (20,0) -| (27,1) -| (20,0);
     \fill [color=black!10] (22,0) -| (27,1) -| (22,0);
    \draw [color=black!70] (20,0) grid (27,1);
\end{scope}

\draw [thick] (20,0) -| (27,1) -| (20,0);
\draw [very thick, color = red] (22,-1.5) -- (22,3);
\draw [thick] (22,0) -- (23,1);
\draw [thick] (22,1) -- (23,0);

\node at (22.5,1.9) {$\downarrow$};

\node at (22.5, 0.5) {$\theta$};
\node at (24.5, 0.5) {$\theta$};
\node at (25.5, 0.5) {$\theta$};

\node at (20.5,-0.7) {1};
\node at (21.5,-0.7) {2};
\node at (22.5,-0.7) {3};
\node at (23.5,-0.7) {4};
\node at (24.5,-0.7) {5};
\node at (25.5,-0.7) {6};
\node at (26.5,-0.7) {7};

  \begin{scope}
    \clip (30,0) -| (37,1) -| (30,0);
    \fill [color=black!10] (32,0) -| (33,1) -| (32,0);
    \fill [color=black!10] (34,0) -| (37,1) -| (34,0);
    \draw [color=black!70] (30,0) grid (37,1);
\end{scope}

\draw [thick] (30,0) -| (37,1) -| (30,0);
\draw [very thick, color = red] (32,-1.5) -- (32,3);

\node at (33.5,1.9) {$\downarrow$};

\draw [thick] (32,0) -- (33,1);
\draw [thick] (32,1) -- (33,0);

\node at (32.5, 0.5) {$\theta$};
\node at (34.5, 0.5) {$\theta$};
\node at (35.5, 0.5) {$\theta$};

\node at (30.5,-0.7) {1};
\node at (31.5,-0.7) {2};
\node at (32.5,-0.7) {3};
\node at (33.5,-0.7) {4};
\node at (34.5,-0.7) {5};
\node at (35.5,-0.7) {6};
\node at (36.5,-0.7) {7};

\end{tikzpicture}

\begin{tikzpicture}[scale = 0.4]
  \begin{scope}
    \clip (0,0) -| (7,1) -| (0,0);
    \fill [color=black!10] (2,0) -| (3,1) -| (2,0);
    \fill [color=black!10] (4,0) -| (7,1) -| (4,0);
    \draw [color=black!70] (0,0) grid (7,1);
\end{scope}

\node at (4.5,1.9) {$\downarrow$};

\draw [thick] (4,0) -- (5,1);
\draw [thick] (4,1) -- (5,0);

\node at (2.5, 0.5) {$\theta$};
\node at (4.5, 0.5) {$\theta$};
\node at (5.5, 0.5) {$\theta$};

\node at (0.5,-0.7) {1};
\node at (1.5,-0.7) {2};
\node at (2.5,-0.7) {3};
\node at (3.5,-0.7) {4};
\node at (4.5,-0.7) {5};
\node at (5.5,-0.7) {6};
\node at (6.5,-0.7) {7};

\draw [thick] (0,0) -| (7,1) -| (0,0);
\draw [very thick, color = red] (2,-1.5) -- (2,3);

  \begin{scope}
    \clip (10,0) -| (17,1) -| (10,0);
    \fill [color=black!10] (12,0) -| (13,1) -| (12,0);
    \fill [color=black!10] (14,0) -| (17,1) -| (14,0);
    \draw [color=black!70] (10,0) grid (17,1);
\end{scope}

\draw [thick] (10,0) -| (17,1) -| (10,0);
\draw [very thick, color = red] (12,-1.5) -- (12,3);

\node at (15.5,1.9) {$\downarrow$};

\draw [thick] (15,0) -- (16,1);
\draw [thick] (15,1) -- (16,0);

\node at (12.5, 0.5) {$\theta$};
\node at (14.5, 0.5) {$\theta$};
\node at (15.5, 0.5) {$\theta$};

\node at (10.5,-0.7) {1};
\node at (11.5,-0.7) {2};
\node at (12.5,-0.7) {3};
\node at (13.5,-0.7) {4};
\node at (14.5,-0.7) {5};
\node at (15.5,-0.7) {6};
\node at (16.5,-0.7) {7};

  \begin{scope}
    \clip (20,0) -| (27,1) -| (20,0);
    \fill [color=black!10] (22,0) -| (23,1) -| (22,0);
    \fill [color=black!10] (24,0) -| (26,1) -| (24,0);
    \draw [color=black!70] (20,0) grid (27,1);
\end{scope}

\draw [thick] (20,0) -| (27,1) -| (20,0);
\draw [very thick, color = red] (22,-1.5) -- (22,3);

\node at (26.5,1.9) {$\downarrow$};

\draw [thick] (25,0) -- (26,1);
\draw [thick] (25,1) -- (26,0);

\node at (22.5, 0.5) {$\theta$};
\node at (24.5, 0.5) {$\theta$};
\node at (25.5, 0.5) {$\theta$};

\node at (20.5,-0.7) {1};
\node at (21.5,-0.7) {2};
\node at (22.5,-0.7) {3};
\node at (23.5,-0.7) {4};
\node at (24.5,-0.7) {5};
\node at (25.5,-0.7) {6};
\node at (26.5,-0.7) {7};

\end{tikzpicture}
\end{center}

\caption{The pointer construction for the superspace elements $q_{J,i} \in \Omega_n$ and the polynomials
$p_{J,i} \in \CC[\xx_n]$. Here $n = 7$ and $J = \{3,5,6\}$. Boxes whose positions in $J$ are indicated with a $\theta$.
Shaded boxes indicate the set of bosonic variables involved at each stage; boxes with a $\theta$ are always shaded.
The degree of the $h$-polynomial in $q_{J,i}$ and $p_{J,i}$ is the number of unshaded boxes, plus one.
Once the pointer crosses the red line 
(i.e. reaches the minimum element of $J$), the definition of $q_{J,i}$ and $p_{J,i}$ involves derivatives. The pointer points 
to shaded boxes to the left of the right line, and an unshaded box or $\theta$ box to the right of the red line. The $\theta$ decoration
with an $\times$ corresponds to an unused $\theta$-variable $\theta_s$ in the case of $q_{J,i}$, or a partial derivative $\partial_s$
in the case of $p_{J,i}$. The $\times$ appears on the closest $\theta$ which is weakly to the left of the pointer.}
\label{pointer-figure}
\end{figure}
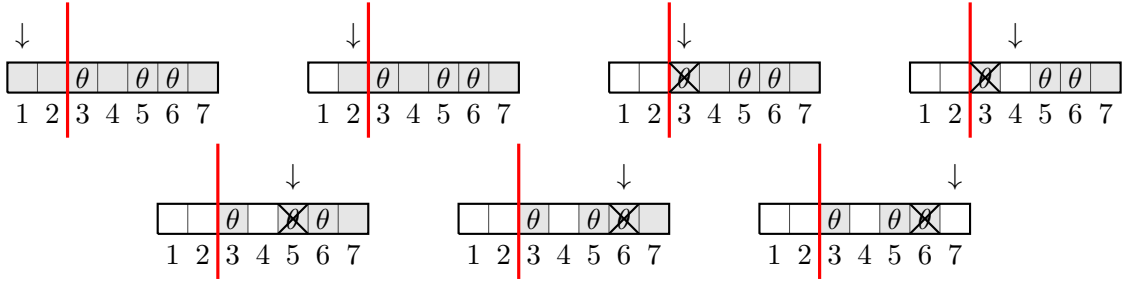

\begin{lemma}
\label{q-polynomial-observations}
Let $J \subseteq [n]$ and let $(q_{J,1}, q_{J,2}, \dots, q_{J,n})$ be the associated sequence of elements of $\Omega_n$.
For any $1 \leq i \leq n$, the superspace element $q_{J,i}$ satisfies the following properties.
\begin{enumerate}
\item We have $q_{J,i} \in SI_n$.
\item The superspace element $q_{J,i}$ is bihomogeneous with fermionic degree $|J|$ and bosonic degree $\stair(J)_i$ where
$\stair(J) = (\stair(J)_1, \dots, \stair(J)_n)$ is the $J$-staircase.
\item The element $q_{J,i}$ lies in the subspace $\bigoplus_{J \leq_{\Gale} K} \CC[\xx_n] \cdot \theta_K$ of $\Omega_n$ spanned by monomials whose
fermionic parts are greater than or equal to $J$ in Gale order.
\end{enumerate}  
\end{lemma}

\begin{proof}
The memberships \eqref{classical-coinvariant-ideal-membership} and \eqref{supercoinvariant-ideal-membership} and the construction
of $q_{J,i}$ imply (1).
Moving the pointer from $i-1$ to $i$ does not change the bosonic degree of $q_{J,i}$ when the box $i$ is decorated with a 
$\theta$, and increases the bosonic degree of $q_{J,i}$ by 1 otherwise, so (2) also holds by construction. To see why (3) is true,
observe that the only surviving fermionic monomials $\theta_K$ in the expression
\begin{multline}
dh_r(J \cup \{i+1, \dots, n\}) \cdot \theta_{J - \max(J \cap \{1, \dots, i\})} = \\ 
\sum_{k \in J \cup \{i+1, \dots, n \}}  \partial_k h_r(J \cup \{i+1, \dots, n\}) \cdot \theta_k \cdot \theta_{J - \max(J \cap \{1, \dots, i\})}
\end{multline}
satisfy $J \leq_\Gale K$.
\end{proof}

We will be interested in the projections of the $q_{J,i}$ to $\CC[\xx_n] \cdot \theta_J$.
To this end, define polynomials $(p_{J,1}, p_{J,2}, \dots, p_{J,n}) \in \CC[\xx_n]$ by the rule
\begin{equation}
p_{J,i} =
\begin{cases}
 h_i(i, i+1, \dots, n\}) & j < \min(J) \\
\partial_s( h_r(J \cup \{i+1, \dots, n\})) & s = \max(I \cap \{1, \dots, i\})
 \end{cases}
\end{equation}
where (as in the definition of $q_{J,i}$) in the second branch $r := n - |J \cup \{i+1, \dots, n\}| + 1$.
As with the superspace elements $q_{J,i}$, the polynomials $p_{J,i}$ are easily visualized using the pointer construction.
The index $s$ on the partial derivative operator $\partial_s$ is the maximal element of $j$ weakly to the left of the pointer.
As the pointer moves from left to right, the degree of the $h$-polynomial increases and its number of arguments decreases.
When $n = 7$ and $J = \{3,5,6\}$, Figure~\ref{pointer-figure} yields
\begin{equation*}
p_{J,1} = h_1(1234567) \quad 
p_{J,2} = h_2(234567) \quad 
p_{J,3} = \partial_3 h_3(34567) \quad
p_{J,4} = \partial_3 h_4(3567)
\end{equation*}
\begin{equation*}
p_{J,5} = \partial_5 h_4(3567) \quad 
p_{J,6} = \partial_6 h_4(3567) \quad
p_{J,7} = \partial_6 h_5(356).
\end{equation*}

By Lemma~\ref{q-polynomial-observations} (3), we have
\begin{equation}
\label{q-p-equivalence}
q_{J,i} \equiv p_{J,i} \cdot \theta_J  \mod \bigoplus_{J <_\Gale K} \CC[\xx_n] \cdot \theta_K
\end{equation}
for all subsets $J \subseteq [n]$ and $1 \leq i \leq n$.
The polynomials $p_{J,i} \in \CC[\xx_n]$ 
are the `Gale-leading terms' of the $q_{J,i} \in \Omega_n$
and will give us access to the tools of classical commutative algebra in 
$\CC[\xx_n]$.  
In particular, we will prove that $p_{J,1}, \dots, p_{J,n}$ is a regular sequence in $\CC[\xx_n]$ as long as $1 \notin J$.
Our first step in doing so is an identity involving partial derivatives of homogeneous symmetric polynomials
in partial variable sets.

\begin{lemma}
\label{miracle-identity}
If $S \subseteq [n]$ is any subset with $a,b \in S$ and $c \notin S$ 
then
\begin{equation}
\label{differential-h-identity}
\partial_a h_r(S) =
\partial_b h_r(S) + (x_c - x_b) \cdot \partial_b h_{r-1}(S \cup c) -
(x_c - x_a) \cdot  \partial_a h_{r-1}(S \cup c)
\end{equation}
for all $r > 1$. 
\end{lemma}

In Lemma~\ref{miracle-identity} we allow the possibility $a = b$, in which case the claimed equation is trivial.

\begin{proof}
The RHS of 
Equation~\eqref{differential-h-identity} may be expanded and regrouped to give
\begin{multline}
\label{script-one}
\partial_b h_r(S) + (x_c - x_b) \partial_b h_{r-1}(S \cup c) -
(x_c - x_a) \partial_a h_{r-1}(S \cup c)  = \\
\left[ \partial_b (h_r(S) + x_c h_{r-1}(S \cup c)) - \partial_a(x_c h_{r-1}(S \cup c)) \right]  - [x_b \partial_b h_{r-1}(S \cup c)] + 
[x_a \partial_a h_{r-1}(S \cup c)].
\end{multline}
Since $h_r(S) + x_c h_{r-1}(S \cup c) = h_r(S \cup c)$, the expression in the first set of brackets $[ \, \cdots ]$ on the 
RHS of Equation~\eqref{script-one} equals
$[\partial_b h_r(S \cup c) - \partial_a h_r(S \cup c) + \partial_a h_r(S)]$, the expression in the second set of brackets equals
$[ \partial_b (x_b h_{r-1}(S \cup c)) - h_{r-1}(S \cup c) ]$, and the expression in the third set of brackets equals
$ [ \partial_a (x_a h_{r-1}(S \cup c)) - h_{r-1}(S \cup c)]$.  Plugging all this in yields
\begin{multline}
\label{script-two}
\left[ \partial_b (h_r(S) + x_c h_{r-1}(S \cup c)) - \partial_a(x_c h_{r-1}(S \cup c)) \right]  - [x_b \partial_b h_{r-1}(S \cup c)] + 
[x_a \partial_a h_{r-1}(S \cup c)] \\  = 
[\partial_b h_r(S \cup c) - \partial_a h_r(S \cup c) + \partial_a h_r(S)] -
[ \partial_b (x_b h_{r-1}(S \cup c)) - \cancel{h_{r-1}(S \cup c)} ] \\ + 
 [ \partial_a (x_a h_{r-1}(S \cup c)) - \cancel{h_{r-1}(S \cup c)}]
\end{multline}
with the indicated cancellations.
After performing these cancellations, the RHS of Equation~\eqref{script-two} may be regrouped as 
\begin{multline}
\label{script-three}
[\partial_b h_r(S \cup c) - \partial_a h_r(S \cup c) + \partial_a h_r(S)] -
[ \partial_b (x_b h_{r-1}(S \cup c))]  + 
 [ \partial_a (x_a h_{r-1}(S \cup c))] \\ 
 = \partial_a h_r(S) + \left\{  \partial_b (h_r(S \cup c) - x_b h_{r-1}(S \cup c))  \right\} -
 \left\{ \partial_a(h_r(S \cup c) - x_a h_{r-1}(S \cup c))   \right\}.
\end{multline}
Since the expression $h_r(S \cup c) - x_b h_{r-1}(S \cup c) = h_r( (S \cup c) - b )$ is independent of $x_b$,
the partial derivative $\partial_b$ 
in the first set of curly braces $\{ \, \cdots \}$ on the RHS of Equation~\eqref{script-three} vanishes; the expression
in the second set of curly braces vanishes for similar reasons.
This completes the proof of Equation~\eqref{differential-h-identity}.
\end{proof}

The polynomial identity in Lemma~\ref{miracle-identity} is, to the authors, somewhat miraculous; it would be nice
to have a conceptual understanding of ``why" it should be true.
We use this identity to show that the ideal $\III_J$ generated by the polynomials $p_{J,1}, \dots, p_{J,n} \in \CC[\xx_n]$
contains certain strategic partial derivatives.

\begin{lemma}
\label{script-ideal-containment}
Let $J \subseteq [n]$ and write $\III_J = (p_{J,1}, \dots, p_{J,n}) \subseteq \CC[\xx_n]$ 
for the ideal generated by $p_{J,1}, \dots, p_{J,n}$.
 For any index $j \in J$, we have $\partial_j h_{n-|J|+1}(J) \in \III_J$.
\end{lemma}

\begin{proof}
We prove the following claim, which is stronger than the lemma and amenable to induction.

{\bf Claim:}
{\em The polynomials in question lie in the ideal
\begin{equation}
\III'_J := (p_{J,j_0}, p_{J,j_0 + 1}, \dots, p_{J,n}) \subseteq \CC[x_{j_0}, x_{j_0 + 1}, \dots, x_n]
\end{equation}
where $j_0 = \min(J)$ is the smallest element of $J$.}

The pointer construction makes it clear that  the generators of $\III'_J$ do not involve the variables $x_1, x_2, \dots, x_{j_0 - 1}$
and so lie in the polynomial ring $\CC[x_{j_0}, x_{j_0 + 1}, \dots, x_n]$ generated by the remaining variables.
We prove the Claim by induction on the number $n - j_0 + 1$ of variables in the ambient ring of $\III'_J$.

If $J = \{n-r+1, \dots, n-1, n\}$ is a terminal subset of $[n]$, the polynomials in the Claim are generators of the ideal $\III'_J$.
 Furthermore, for any subset $J \subseteq [n]$, if $j = \max(J)$ is the largest element of $J$, then 
$\partial_j h_{n-|J|+1}(J) = p_{J,n}$ is also a generator of $\III'_J$.

By the above paragraph, we may assume that $j_0 = \min(J) \neq \max(J)$ and that there exists an element 
$c \in [n] - J$ with $c > j_0$. Let $c_0 := \min \{ j_0 < c \leq n \,:\, c \notin J \}$ be the smallest such $c$ and define $S \subseteq [n]$
by
\begin{equation}
S := \{ j_0, j_0 + 1, \dots , n-1, n \} - \{ c_0 \}.
\end{equation}
Observe that the elements $j_0, j_0 + 1, \dots, c_0 - 2, c_0 - 1$ of $S$ lie in $J$.
Let $r := n - |S| + 1$. We apply Lemma~\ref{miracle-identity} iteratively as follows.
\begin{itemize}
\item Since $\partial_{c_0 - 1} h_r(S), \partial_{c_0 - 1} h_{r-1}(S \cup c_0), \partial_{c_0 - 2}(S \cup c_0) \in \III'_J$, 
Lemma~\ref{miracle-identity} with $a = c_0 - 2, b = c_0 - 1$, and $c = c_0$ implies $\partial_{c_0 - 2} h_r(S) \in \III'_J$.
\item Since $\partial_{c_0 - 2} h_r(S), \partial_{c_0 - 2} h_{r-1}(S \cup c_0), \partial_{c_0 - 3}(S \cup c_0) \in \III'_J$,
Lemma~\ref{miracle-identity} with $a = c_0 - 3, b = c_0 - 2,$ and $c = c_0$ implies $\partial_{c_0 - 3} h_r(S) \in \III'_J$.
\item Since $\partial_{c_0 - 3} h_r(S), \partial_{c_0 - 3} h_{r-1}(S \cup c_0), \partial_{c_0 - 4}(S \cup c_0) \in \III'_J$,
Lemma~\ref{miracle-identity} with $a = c_0 - 3, b = c_0 - 2,$ and $c = c_0$ implies $\partial_{c_0 - 4} h_r(S) \in \III'_J$, and so on.
\end{itemize}
We see that the polynomials
\begin{equation}
p'_{J,j_0} := \partial_{j_0} h_r(S)  \quad
p'_{J,j_0 + 1} := \partial_{j_0 + 1} h_r(S) \quad  \dots  \quad 
p'_{J,c_0 - 1} := \partial_{c_0 - 1} h_r(S)
\end{equation}
lie in $\III'_J$ so that 
\begin{equation}
\label{half-prime-containment}
(p'_{J,j_0}, p'_{J,j_0 + 1}, \dots, p'_{J,c_0 - 1}, p_{J,c_0 + 1}, p_{J, c_0 + 2}, \dots, p_{J, n}) \subseteq \III'_J
\end{equation}
as ideals in $\CC[x_{j_0}, x_{j_0 + 1}, \dots, x_n]$.
But the generators on the ideal on the LHS of \eqref{half-prime-containment} do not involve the variable $x_{c_0}$.
In fact, if we consider the variable set
\begin{equation}
\xx := (x_{j_0}, x_{j_0 + 1}, \dots, x_{c_0 - 1}, x_{c_0 + 1}, \dots, x_{n-1}, x_n)
\end{equation}
obtained from
 our old variable set $(x_{j_0}, x_{j_0 + 1}, \dots, x_n)$ by removing $x_{c_0}$, then
 \begin{equation}
(p'_{J,j_0}, p'_{J,j_0 + 1}, \dots, p'_{J,c_0 - 1}, p_{J,c_0 + 1}, p_{J, c_0 + 2}, \dots, p_{J, n}) = \III'_{J'}
\end{equation}
as ideals in $\CC[\xx]$ where $J' = (J - j_0) \cup c_0$ is the corresponding cyclic rotation of the set $J$.
Since the variable set $\xx$ contains fewer variables than the original set $\{ x_{j_0}, x_{j_0 + 1}, \dots, x_n \}$, 
we are done by induction.
\end{proof}

An example may help clarify Lemma~\ref{script-ideal-containment} and its proof.  Suppose $n = 7$ and $J = \{3,5,6\}$.
We have $\III_J = (p_{J,1}, \dots, p_{J,7})$ where 
\begin{equation*}
p_{J,1} = h_1(1234567) \quad 
p_{J,2} = h_2(234567) \quad 
p_{J,3} = \partial_3 h_3(34567) \quad
p_{J,4} = \partial_3 h_4(3567)
\end{equation*}
\begin{equation*}
p_{J,5} = \partial_5 h_4(3567) \quad 
p_{J,6} = \partial_6 h_4(3567) \quad
p_{J,7} = \partial_6 h_5(356).
\end{equation*}
Our aim is to show that the ideal $\III_J$ contains the elements
\begin{equation*}
\partial_3 h_5(356), \quad \partial_5 h_5(356), \quad \partial_6 h_5(356).
\end{equation*}
To this end, we reason as follows.
\begin{itemize}
\item The element $\partial_6 h_5(356) = p_{J,7}$ is a generator of $\III_J$. This was one of the desired memberships.
\item Since $\partial_3 h_3(34567) = p_{J,3}, \partial_3 h_4(3567) = p_{J,4},$ and $\partial_5 h_4(3567) = p_{J,5}$ are elements
of $\III_J$, Lemma~\ref{miracle-identity} with $S = \{3,5,6,7\}, a = 3, b = 5,$ and $c = 4$ implies
$\partial_3 h_4(3567) \in \III_J$.
\item Since $\partial_3 h_4(3567), \partial_6 h_4(3567) = p_{J,6},$ and $\partial_6 h_5(356)$ are elements of $\III_J$,
Lemma~\ref{miracle-identity} with $S = \{3,5,6\}, a = 3, b = 6,$ and $c = 7$ implies
$\partial_3 h_5(356) \in \III_J$. This was one of the desired memberships.
\item Since $\partial_5 h_4(3567) = p_{J,5}, \partial_6 h_4(3567) = p_{J,6}, \partial_6 h_5(356) \in \III_J$,
Lemma~\ref{miracle-identity} with $S = \{3,5,6\}, a = 5, b = 6,$ and $c = 7$ implies
$\partial_5 h_5(356) \in \III_J$. This was the remaining desired membership.
\end{itemize}
Observe that we did not use the generators $p_{J,1}, p_{J,2} \in \III_J$ to derive these memberships, so that in fact we showed
membership in the smaller ideal
\begin{equation*}
\III'_J = (p_{J,3}, p_{J,4}, p_{J,5}, p_{J,6}, p_{J,7}) \subseteq \CC[x_3,x_4,x_5,x_6,x_7].
\end{equation*}

\begin{lemma}
\label{regular-sequence-lemma}
Let $J \subseteq [n]$ with $\stair(J) = (\stair(J)_1, \dots, \stair(J)_n)$. If $1 \notin J$,
the sequence of polynomials $p_{J,1},  \dots, p_{J,n}$ is a regular sequence in $\CC[\xx_n]$
of homogeneous degrees $\stair(J)_1, \dots, \stair(J)_n$.
\end{lemma}

If $1 \in J$, then $p_{J,1} = \partial_1 h_1(x_1, \dots, x_n) = \partial_1 (x_1 + \cdots + x_n) = 1$
 is a unit in $\CC[\xx_n]$. Correspondingly, we have
$\stair(J)_1 = 0$.
Since members of regular sequences are required to be of positive homogeneous degree, we must exclude this 
case from Lemma~\ref{regular-sequence-lemma}.

\begin{proof}
Since $1 \notin J$, the sequence $\stair(J)$ has positive entries.
The assertion on degrees is Lemma~\ref{q-polynomial-observations} (2).
As in Lemma~\ref{script-ideal-containment}, let 
$\III_J = ( p_{J,1}, \dots, p_{J,n} ) \subseteq \CC[\xx_n]$.
By Lemma~\ref{regular-sequence-criterion}, it is enough to
show that the variety $\VV(\III) \subseteq \CC^n$
cut out by $\III$ consists of $\{ 0 \}$ alone.
We use elimination to focus on coordinates in $\CC^n$ indexed by $J$.

Swanson and Wallach proved \cite[Lem. 6.2]{SW2} that
that the polynomials $\partial_j h_{n-|J|+1}(J)$ for $j \in J$
have no common zero in $\CC^J$. By 
Lemma~\ref{script-ideal-containment}, for any locus point $a = (a_1, \dots, a_n) \in \VV(\III_J)$, 
we must 
have $a_j = 0$ for any $j \in J$. Setting the variables $\{x_j \,:\, j \in J \}$ to zero in the remaining polynomials
\begin{equation}
p_{J,i} \mid_{x_j \rightarrow 0 \text{ for } j \in J}  \quad \quad (i \notin J)
\end{equation}
gives a sequence of positive degree homogeneous polynomials in $\CC[x_i \,:\, i \notin J]$ which are easily seen to be triangular.
We conclude that $a_i = 0$ for $i \notin J$, so that $a = 0$ as required.
\end{proof}

Lemma~\ref{regular-sequence-lemma} implies that the quotient ring $\CC[\xx_n]/(p_{J,1}, \dots, p_{J,n})$ has Hilbert series 
\begin{equation}
\Hilb( \CC[\xx_n]/(p_{J,1}, \dots, p_{J,n}); q) = 
[\stair(J)_1]_q \cdots [\stair(J)_n]_q.
\end{equation}
This formula remains true when $1 \in J$, for then $p_{J,1} = 1$ and $\CC[\xx_n]/(p_{J,1}, \dots, p_{J,n}) = 0$. 
In particular, there exists a set $\BBB_n(J) \subseteq \CC[\xx_n]$ of homogeneous polynomials with degree generating 
function
$[\stair(J)_1]_q \cdots [\stair(J)_n]_q$ such that $\BBB_n(J)$ descends to a vector space basis of 
$\CC[\xx_n]/(p_{J,1}, \dots, p_{J,n})$.

\subsection{An abstract straightening lemma}
The proof of Lemma~\ref{regular-sequence-lemma}  relied on a
a tricky induction in Lemma~\ref{script-ideal-containment}
and miraculous polynomial identity in Lemma~\ref{miracle-identity}. Our next result should  persuade the reader that 
Lemma~\ref{regular-sequence-lemma} was worth the effort.

\begin{lemma}
\label{straightening-lemma}
{\em (Straightening)}
Let $J \subseteq [n]$ with $\stair(J) = (\stair(J)_1, \dots, \stair(J)_n)$.
There exists a finite set $\BBB_n(J) \subseteq \CC[\xx_n]$ of nonzero homogeneous polynomials with 
degree generating function
\begin{equation}
\sum_{m \in \BBB_n(J)} q^{\deg(m)} = [\stair(J)_1]_q [\stair(J)_2]_q \cdots [\stair(J)_n]_q
\end{equation}
such that for any polynomial $f \in \CC[\xx_n]$ we have an expression of the form
\begin{equation}
f \cdot \theta_J = \left(\sum_{m\in \BBB_n(J)} c_{f,m} \cdot m \cdot \theta_J \right) + g + \Sigma
\end{equation}
where
\begin{itemize}
\item the $c_{f,m} \in \CC$ are constants which depend on $f$ and $m$,
\item the element $g \in SI_n$ lies in the supercoinvariant ideal, and
\item the ``error term" $\Sigma$ lies in $\bigoplus_{J <_{\Gale} K} \CC[\xx_n] \cdot \theta_K$.
\end{itemize}
\end{lemma}

\begin{proof}
As explained after Lemma~\ref{regular-sequence-lemma}, there exists a set $\BBB_n(J) \subseteq \CC[\xx_n]$ of homogeneous
polynomials with the given degree generating function which descends to a vector space basis of 
$\CC[\xx_n]/(p_{J,1}, \dots, p_{J,n})$. We prove that $\BBB_n(J)$ satisfies the conditions of the lemma.

The given polynomial $f \in \CC[\xx_n]$ may be written as
\begin{equation}
\label{straighten-one}
f  = \left( \sum_{m \in \BBB_n(J)} c_{f,m} \cdot m \right) + \sum_{j = 1}^n A_j \cdot p_{J,j}
\end{equation}
for some scalars $c_{f,m} \in \CC$ and polynomials $A_j \in \CC[\xx_n]$.
Multiplying both sides of Equation~\eqref{straighten-one} by $\theta_J$ yields 
\begin{equation}
\label{straighten-two}
f \cdot \theta_J  = \left( \sum_{m \in \BBB_n(J)} c_{f,m} \cdot m \cdot \theta_J \right) + \sum_{j = 1}^n A_j \cdot p_{I,j} \cdot \theta_J.
\end{equation}
Equation~\eqref{q-p-equivalence} gives the relation
\begin{equation}
\label{straighten-three}
f \cdot \theta_J \equiv \left( \sum_{m \in \BBB_n(J)} c_{f,m} \cdot m \cdot \theta_J \right) + \sum_{j = 1}^n A_j \cdot q_{J,j}
\mod \bigoplus_{J <_\Gale K} \CC[\xx_n] \cdot \theta_K
\end{equation}
modulo the linear subspace 
$\bigoplus_{J <_\Gale K} \CC[\xx_n] \cdot \theta_K$ of $\Omega_n$. Finally,
Lemma~\ref{q-polynomial-observations} (1) implies the membership 
$g := \sum_{j = 1}^n A_j \cdot q_{J,j} \in SI_n$, which completes the proof.
\end{proof}

Lemma~\ref{straightening-lemma} implies that the set $\BBB_n \subseteq \Omega_n$ of superspace elements given by
\begin{equation}
\label{B-spanning-set}
\BBB_n := \bigsqcup_{J \subseteq [n]} \BBB_n(J) \cdot \theta_J
\end{equation}
descends to a spanning set in $SR_n$.
Indeed, if this were not the case, let $J \subseteq [n]$ be a Gale-maximal subset such that $f \cdot \theta_J \in \Omega_n$ 
does not lie in the span 
of $\BBB_n$ modulo $SI_n$ for some $f \in \CC[\xx_n]$.  Lemma~\ref{straightening-lemma} implies
that 
\begin{equation}
f \cdot \theta_J \equiv \left(\sum_{m\in \BBB_n(J)} c_{f,m} \cdot m \cdot \theta_J \right) + \Sigma \quad \mod SI_n
\end{equation}
for some constants $c_{f,m} \in \CC$ where $\Sigma \in \bigoplus_{J <_{\Gale} K} \CC[\xx_n] \cdot \theta_K$.
The term in the parentheses certainly lies in the span of $\BBB_n$. The Gale-maximality of $J$ implies that 
$\Sigma$ lies in the span of $\BBB_n$, as well, giving a contradiction.

The straightening result of Lemma~\ref{straightening-lemma} is rather abstract in that it does not give a 
 formula for the polynomials in $\BBB_n(J)$. 
 While any generic set of polynomials of the appropriate degrees will do, the authors are unaware of an explicit formula
 for the set $\BBB_n(J)$. In general, objects related to $SR_n$ have resisted analysis by Gr\"obner-theoretic techniques,
 which is reflected in the abstract statement of Lemma~\ref{straightening-lemma}.

 Lemma~\ref{straightening-lemma} implies an upper bound for the bigraded Hilbert series
of $SR_n$.  Given two polynomials $f(q,z), g(q,z)$ in variables $q, z$, we write $f \leq g$ to mean
 that $g - f$ is a polynomial in $q,z$ with nonnegative coefficients.

\begin{proposition}
\label{upper-bound-dimension}
The bigraded Hilbert series $\Hilb(SR_n;q,z)$ is bounded above by
\begin{equation}
\Hilb(SR_n;q,z) \leq \sum_{J \subseteq [n]} z^{|J|} \sum_{f \in \BBB_n(J)} q^{\deg(f)} = \sum_{k = 1}^n z^{n-k} \cdot [k]!_q \cdot \Stir_q(n,k).
\end{equation}
\end{proposition}

\begin{proof}
As explained above, Lemma~\ref{straightening-lemma} implies that $\BBB_n = \bigsqcup_{J \subseteq [n]} \BBB_n(J)$ descends
 to a spanning set of $SR_n$. Since $\sum_{m \in \BBB_n(J)} q^{\deg(m)} = [\stair(J)_1]_q \cdots [\stair(J)_n]_q$,
 the result follows from Lemma~\ref{ss-dimension}.
\end{proof}

\section{Differential operators and colon ideals}
\label{Differential}

The straightening result of Lemma~\ref{straightening-lemma} led to the upper bound on
the dimension of $SR_n$ in Proposition~\ref{upper-bound-dimension}.
Our next task is to bound this dimension from below.
To this end, we define strategic differential operators $\DD_J$ whose action on $\CC[\xx_n]$ has Gale maximum term $\theta_J$.
Analysis of this leading term will lead to finding a lower bound for quotient rings of the form $\CC[\xx_n] / (I_n : f_J)$
where  $I_n \subseteq \CC[\xx_n]$ is the classical coinvariant ideal and 
the $f_J \in \CC[\xx_n]$ are products of linear forms determined by $\DD_J$.
It will turn out (Theorem~\ref{colon-ideal-identification}) that $(I_n : f_J)$ is generated by the regular sequence
$p_{J,1}, \dots , p_{J,n}$ of Lemma~\ref{regular-sequence-lemma}. Together with the triangularity property of the $\DD_J$,
this will lead to the required lower bound on $SR_n$.

\subsection{The differential operators $\DD_J$}
Let $\HHH$ be the $n \times n$ matrix of complete homogeneous symmetric polynomials whose row $i$, column $j$
entry is given by
\begin{equation}
\HHH :=   \left(   h_{i-j}(x_i, x_{i+1}, \dots, x_n)   \right)_{\substack{1 \leq i \leq n \\ 1 \leq j \leq n}}.
\end{equation}
We have $h_0 = 1$ and interpret $h_{j-i} = 0$ whenever $i > j$, so the matrix $\HHH$ is lower triangular with 1's on the diagonal.
We use the matrix $\HHH$ to define a family of differential operators as follows.
Given a subset $K \subseteq [n]$, we introduce the `reversal' notation
\begin{equation}
K^* := \{ n-k+1 \,:\, k \in K \}.
\end{equation}

\begin{defn}
\label{d-operator-definition}
For any subset $J \subseteq [n]$, define a differential operator $\DD_J$ acting on $\Omega_n$ by
\begin{equation}
\DD_J(f) := \sum_{|I| = |J|} (-1)^{\sum I} \Delta_{[n] - J, ([n] - I)^*}(\HHH) \odot d_I (f)
\end{equation}
where $\Delta_{[n] - J, [n] - I}(\HHH) \in \CC[\xx_n]$ is the minor of $\HHH$ with row set $[n] - J$ 
and column set $([n] - I)^*$.
\end{defn}

Since the matrix $\HHH$ is lower triangular, the coefficient of $d_I$ in $\DD_J$ is zero
unless we have $I^* \leq_\Gale J$ in Gale order.
As an example, when $n = 3$ the matrix $\HHH$ is given by 
\begin{scriptsize}
\begin{equation*}
\HHH = 
\begin{pmatrix}
1 & 0 & 0  \\
x_2 + x_3 &  1 & 0  \\
x_3^2 & x_3 & 1
\end{pmatrix}
\end{equation*}
\end{scriptsize}
and we have the differential operators
\begin{scriptsize}
\begin{align*}
\DD_{12}(f) &= - \Delta_{3,1}(\HHH) \odot d_{12}(f) + \Delta_{3,2}(\HHH) \odot d_{13}(f) - \Delta_{3,3}(\HHH) \odot d_{23}(f) \\
\DD_{13}(f) &= - \Delta_{2,1}(\HHH) \odot d_{12}(f) + \Delta_{2,2}(\HHH) \odot d_{13}(f) - \cancel{\Delta_{2,3}(\HHH)} \odot d_{23}(f) \\
\DD_{23}(f) &= - \Delta_{1,1}(\HHH) \odot d_{12}(f) + \cancel{\Delta_{1,2}(\HHH)} \odot d_{13}(f) - \cancel{\Delta_{1,3}(\HHH)} \odot d_{23}(f) 
\end{align*}
\end{scriptsize}
acting on superspace elements $f \in \Omega_3$ where the indicated minors of $\HHH$ vanish for support reasons.
Applying the formula 
$d_i(f) = ( x_1^i \odot f ) \theta_1 + (x_2^i \odot f) \theta_2 + (x_3^i \odot f) \theta_3$, these operators may be expressed in 
the more illuminating form
\begin{scriptsize}
\begin{align*}
\DD_{12}(f) &= 
( x_1 (x_1 - x_2) (x_1 - x_3) x_2 (x_2 - x_3) )
\odot f \cdot  \theta_1 \theta_2 \\
\DD_{13}(f) &=
(   x_1^2 x_2^2 + x_1^2 x_2 x_3 - x_1 x_2^2 x_3 - x_1^3 x_3  )
 \odot f \cdot \theta_1 \theta_2 -
 (x_1 (x_1 - x_2) (x_1 - x_3) x_3) 
 \odot f \cdot \theta_1 \theta_3 \\
\DD_{23}(f) &=
(x_1^2 x_2 - x_1 x_2^2) \odot f \cdot \theta_1 \theta_2 +
(x_1^2 x_3 - x_1 x_3^2) \odot f \cdot \theta_1 \theta_3 +
(x_2 (x_2 - x_3) x_3) \odot f \cdot \theta_2 \theta_3
\end{align*}
\end{scriptsize}
which reveals a triangularity property with respect to the fermionic monomials 
$\theta_1 \theta_2, \theta_1 \theta_3,$ and $\theta_2 \theta_3$.
Furthermore, the `leading coefficient' $\theta_J$ involved in $\DD_J$ has the form
$f_J \odot (-)$ up to a sign where the polynomials $f_J$ were defined in the introduction.
We will show that this is a general phenomenon.
Our first lemma in this direction is a simple result on the application of the $d_I$ operator to polynomials in $\CC[\xx_n]$;
its proof is left to the reader.

\begin{lemma}
\label{image-of-derivative}
Let $f \in \CC[\xx_n]$ be a polynomial and let $I = \{ i_1 < \cdots < i_r \}$ and 
$K = \{k_1 < \cdots < k_r \}$ be two subsets of $[n]$ of the same size.
The coefficient of $\theta_K$ in $d_I(f) \in \Omega_n$ is the determinant of partial derivatives
\begin{equation}
\begin{vmatrix}
\partial_{k_1}^{i_1} f &  & \cdots & &   \partial_{k_1}^{i_r} f \\
& & & & \\
\vdots & & & & \vdots \\
& & & & \\
\partial_{k_r}^{i_1} f &  & \cdots & &   \partial_{k_r}^{i_r} f \\
\end{vmatrix}.
\end{equation}
\end{lemma}

Definition~\ref{d-operator-definition} and Lemma~\ref{image-of-derivative} motivate 
the following family of polynomials $\FF_{J,K} \in \CC[\xx_n]$ indexed by pairs of subsets  $J, K \subseteq [n]$.
The definition of the $\FF_{J,K}$ also involves the matrix $\HHH$.

\begin{defn}
\label{fjk-definition}
Let $J$ and $K$ be two subsets of $[n]$ of the same size. 
Define a polynomial $\FF_{J,K} \in \CC[\xx_n]$ by
\begin{equation}
\FF_{J,K} :=  \sum_{|I| = |J| = |K|} (-1)^{\sum I} \Delta_{[n] - J, ([n] - I)^*}(\HHH) \cdot \left|  x_k^i   \right|_{k \in K, i \in I} 
\end{equation}
where the row and column indices in the determinant 
$\left|  x_k^i   \right|_{k \in K, i \in I}$ are written in increasing order.
\end{defn}

The differential operators $\DD_J$ and the polynomials $\FF_{J,K}$ are related by
\begin{equation}
\DD_J(f) = \sum_{|K| = |J|}  \left(  \FF_{J,K} \odot f \right) \times \theta_K
\end{equation}
for all $f \in \CC[\xx_n]$.

\begin{remark}
The polynomial $\Delta_{[n] - J, ([n] - I)^*}(\HHH)$ appearing in Definition~\ref{fjk-definition} is (up to variable reversal)
a flagged skew Schur polynomial whose flagging parameter depends on $J$ and whose shape depends on $I$ and $J$,
as may be seen from the Jacobi-Trudi formula.
This is how the $\FF_{J,K}$  were  discovered, but their matrix minor formulation is more convenient for our
purposes.
\end{remark}

We aim to show that the $\FF_{J,K}$ are triangular with respect to Gale order.
As a first step, we express $\FF_{J,K}$ as a single $n \times n$ determinant.

\begin{lemma}
\label{big-determinant}
Let $J = \{ j_1 < \cdots < j_r \}$ and $K = \{ k_1 < \cdots < k_r \}$ be two subsets of $[n]$ of the same size. 
Write $b(J) = (b(J)_1 < b(J)_2 < \cdots )$ for the entries in the complement $[n] - J$ of the set $J$, written in increasing order.
Define an $n \times n$
matrix $A_{J,K}$ in block form
\begin{equation}
A_{J,K} = \begin{pmatrix} B_{J,K} \\ C_{J,K} \end{pmatrix}
\end{equation}
where the top block $B_{J,K}$ has size $r \times n$ and entries
\begin{equation}
B_{J,K} = \begin{pmatrix}  
x_{k_1}^n  &  &  \cdots  & &  x_{k_1}^1 \\
\vdots   & & & & \vdots \\
x_{k_r}^n  & &  \cdots & &  x_{k_r}^1 \\
\end{pmatrix}
\end{equation} 
and the bottom block $C_{J,K}$ has size $(n-r) \times n$ and entries
\begin{equation}
C_{J,K} = ( h_{b(J)_i - j}(x_{b(J)_i}, x_{b(J)_i + 1}, \dots, x_n) )_{1 \leq i \leq n-r, \, \, 1 \leq j \leq n}.
\end{equation}
We have $\FF_{J,K} = \pm \det(A_{J,K})$.
\end{lemma}

\begin{proof}
The determinant $\det(A_{J,K})$ may be evaluated using the rule
\begin{equation}
\label{big-matrix-one}
 \det(A_{J,K}) = 
\sum_{\substack{I \subseteq [n] \\ |I| = r}} (-1)^{\sum I - {r +1 \choose 2}} \cdot 
\Delta_I(B_{J,K}) \cdot \Delta_{[n]-I}(C_{J,K})
\end{equation}
where  $\Delta_I(B_{J,K})$ is the maximal minor of $B_{J,K}$ with column set $I$ and
$\Delta_{[n]-I}(C_{J,K})$ is the maximal minor of $C_{J,K}$ with  complementary column set $[n] - I$.
Now compare with the definition of $\FF_{J,K}$.
\end{proof}

To illustrate Lemma~\ref{big-determinant}, we let $n = 5$, $J = \{1,3\}$, and write $K = \{a,b\}$ for $1 \leq a < b \leq 5$.
Lemma~\eqref{big-determinant} expresses $\FF_{J,K} = \FF_{13,ab}$ as the following $5 \times 5$ determinant.
\begin{scriptsize}
\begin{equation*}
\FF_{13,ab} = \pm
\begin{vmatrix}
x_a^5 & x_a^4 & x_a^3 & x_a^2 & x_a^1 \\
x_b^5 & x_b^4 & x_b^3 & x_b^2 & x_b^1 \\
h_1(2345) & 1 & 0 & 0 & 0 \\
h_3(45) & h_2(45) & h_1(45) & 1 & 0 \\
h_4(5) & h_3(5)  & h_2(5)   & h_1(5) & 1
\end{vmatrix}.
\end{equation*}
\end{scriptsize}

The determinant in Lemma~\ref{big-determinant} may be evaluated to give the desired triangularity 
relation for the polynomials $\FF_{J,K}$.
Lemma~\ref{big-determinant} will also imply that the $\FF_{J,J}$ are given by a  family 
$f_J \in \CC[\xx_n]$ of polynomials defined as follows.

\begin{defn}
For any subset $J \subseteq [n]$, let $f_J \in \CC[\xx_n]$ be the polynomial
\begin{equation}
f_J := \prod_{j \in J} x_j \left( \prod_{i = j+1}^n (x_j - x_i)  \right).
\end{equation}
\end{defn}

Observe that the $f$-polynomial corresponding to a set $J$ factors  $f_J = \prod_{j \in J} f_{\{j\}}$ into $f$-polynomials
corresponding to singletons contained in $J$. 
The polynomials $f_J \in \CC[\xx_n]$ will have deep ties to the supercoinvariant ring $SR_n$.
For later use, we record a criterion for when $f_J$ lies in the classical coinvariant ideal $I_n \subseteq \CC[\xx_n]$.

\begin{lemma}
\label{f-containment-criterion}
Let $J \subseteq [n]$. We have $f_J \in I_n$ if and only if $1 \in J$.
\end{lemma}

\begin{proof}
Suppose $1 \in J$, so that $f_{\{1\}} \mid f_J$.  We claim  $f_{\{1\}} = x_1 (x_1 - x_2) ( x_1 - x_3) \cdots (x_1 - x_n) \in I_n$.
Indeed, if $t$ is a new variable, then modulo $I_n$ we have
\begin{equation}
1 \equiv \frac{1}{(1 - t x_1) (1 - t x_2 ) \cdots (1 - t x_n)} \mod I_n
\end{equation}
so that 
\begin{equation}
(1 - t x_2) \cdots (1 - t x_n) \equiv \frac{1}{1 -t  x_1} \mod I_n
\end{equation}
and taking the coefficient of $t^d$ yields
\begin{equation}
(-1)^d e_d(x_2, \dots, x_n) \equiv x_1^d \mod I_n.
\end{equation}
We conclude that
\begin{equation}
f_{\{1\}} = \sum_{d = 0}^{n-1} (-1)^d e_d(x_2, \dots, x_n) \cdot x_1^{n-d} \equiv n \cdot x_1^n \equiv 0 \mod I_n
\end{equation}
where we used the fact that $x_1^n \in I_n$.

Now suppose $1 \notin J$.  Recall that $\ann_{\CC[\xx_n]}(\delta_n) = I_n$ under the $\odot$-action of $\CC[\xx_n]$ on itself.
Therefore, to show that $f_J \notin I_n$, it is enough to show that $f_J \odot \delta_n \neq 0$. 
Since $f_J = \prod_{j \in J} f_{\{j\}}$, it suffices to show that $f_J \odot \delta_n \neq 0$ when 
$J = J_0 := \{2, 3, \dots, n \}$ is the maximal subset of $[n]$ not containing $1$.
By definition, we have
\begin{equation}
f_{J_0} = (x_2 x_3 \cdots x_n) \times \prod_{2 \leq r <  s \leq n} (x_r - x_s)
\end{equation}
so that the terms of $f_{J_0}$ are (up to a global sign) the terms of
$\delta_n$ in which $x_1$ does not appear. If we use $\doteq$ to denote equality up to a nonzero scalar, we therefore have
\begin{equation}
f_{J_0} \odot \delta_n \doteq f_{J_0} \odot f_{J_0} > 0
\end{equation}
where we used the fact that both $f_{J_0}$ and $\delta_n$ are homogeneous of degree ${n \choose 2}$ and the fact that
$f \odot f > 0$ for any homogeneous nonzero polynomial $f$.
\end{proof}

The determinant in Lemma~\ref{big-determinant} may be evaluated to give the desired triangularity 
relation for the polynomials $\FF_{J,K}$.
Lemma~\ref{big-determinant} will also imply that  $\FF_{J,J} = \pm f_J$.

\begin{lemma}
\label{f-product}
We have $\FF_{J,K} = 0$ unless $J \geq_{\Gale} K$ in Gale order. Furthermore, we have
\begin{equation}
\FF_{J,J} = \pm f_J.
\end{equation}
\end{lemma}

\begin{proof}
We factor $\prod_{k \in K} x_k$ out of 
the upper block $B_{J,K}$ of the determinant $\det(A_{J,K}) = \pm \FF_{J,K}$ in Lemma~\ref{big-determinant}.
Next, we apply column operations to eliminate the $h_d(S)$'s in the bottom portion $C_{J,K}$ of this determinant.

Specifically, we focus on each pivot 1 in $C_{J,K}$ from bottom to top.
Working to the left from a given pivot 1, in row $i$ of $C_{J,K}$, we subtract  $x_c$ times column $j$ of $A_{J,K}$
from column $j-1$, where $x_c$ is a variable belonging to 
$\{x_{b(J)_i}, \dots, x_n \} - \{ x_{b(J)_{i+1}}, \dots, x_n \}$.  Since
$h_d(S) = x_c h_{d-1}(S) + h_d(S - c)$ whenever $c \in S$, this eliminates the $h_d(S)$'s from the bottom portion $C_{J,K}$
of our determinant.
After performing these operations, the determinant $\det(A_{J,K})$ is reduced to a single maximal minor of its (new)
upper portion $B_{J,K}$, from which the result follows.

To see how this works in our  example $J = \{1,3\}$ and $K = \{a, b\}$, we 
factor out $x_a x_b$ from the top two rows of our determinant to get 
\begin{scriptsize}
\begin{equation*}
\begin{vmatrix}
x_a^5 & x_a^4 & x_a^3 & x_a^2 & x_a^1 \\
x_b^5 & x_b^4 & x_b^3 & x_b^2 & x_b^1 \\
h_1(2345) & 1 & 0 & 0 & 0 \\
h_3(45) & h_2(45) & h_1(45) & 1 & 0 \\
h_4(5) & h_3(5)  & h_2(5)   & h_1(5) & 1
\end{vmatrix} =
x_a x_b
\begin{vmatrix}
x_a^4 & x_a^3 & x_a^2 & x_a^1 & 1 \\
x_b^4 & x_b^3 & x_b^2 & x_b^1 & 1 \\
h_1(2345) & 1 & 0 & 0 & 0 \\
h_3(45) & h_2(45) & h_1(45) & 1 & 0 \\
h_4(5) & h_3(5)  & h_2(5)   & h_1(5) & 1
\end{vmatrix}.
\end{equation*}
\end{scriptsize}
Our focus shifts to the bottom three rows.
Since the bottom pivot 1 is in column 5, we subtract $x_5$ times each column from the previous column, resulting in
\begin{scriptsize}
\begin{equation*}
x_a x_b
\begin{vmatrix}
x_a^4 & x_a^3 & x_a^2 & x_a^1 & 1 \\
x_b^4 & x_b^3 & x_b^2 & x_b^1 & 1 \\
h_1(2345) & 1 & 0 & 0 & 0 \\
h_3(45) & h_2(45) & h_1(45) & 1 & 0 \\
h_4(5) & h_3(5)  & h_2(5)   & h_1(5) & 1
\end{vmatrix}
=
x_a x_b
\begin{vmatrix}
x_a^4 - x_a^3 x_5 & x_a^3 - x_a^2 x_5 & x_a^2 - x_a x_5 & x_a^1 - x_5 & 1 \\
x_b^4 - x_b^3 x_5 & x_b^3 - x_b^2 x_5 & x_b^2 - x_b x_5 & x_b^1 - x_5 & 1 \\
h_1(234) & 1 & 0 & 0 & 0 \\
h_3(4) & h_2(4) & h_1(4) & 1 & 0 \\
0 & 0  & 0  & 0 & 1
\end{vmatrix}.
\end{equation*}
\end{scriptsize}
This has the effect of eliminating the argument $x_5$ from the $h$'s.
To eliminate the $x_4$'s from the arguments of the $h$'s,
we focus on the pivot  1 in row 4, column 4.
For each column before column 2, we subtract $x_4$ times the subsequent column.
The result is
\begin{scriptsize}
\begin{equation*}
x_a x_b
\begin{vmatrix}
x_a^4 - x_a^3 x_5 - x_a^3 x_4 + x_a^2 x_4 x_5
& x_a^3 - x_a^2 x_5 - x_a^2 x_4 + x_a x_4 x_5
& x_a^2 - x_a x_5 - x_a x_4 + x_4 x_5
& x_a^1 - x_5  & 1 \\
x_b^4 - x_b^3 x_5 - x_b^3 x_4 + x_b^2 x_4 x_5
& x_b^3 - x_b^2 x_5  - x_b^2 x_4 + x_b x_4 x_5
& x_b^2 - x_b x_5 - x_b x_4 + x_4 x_5
 & x_b^1 - x_5  & 1 \\
h_1(23) & 1 & 0 & 0 & 0 \\
0 & 0 & 0 & 1 & 0 \\
0 & 0  & 0  & 0 & 1
\end{vmatrix}.
\end{equation*}
\end{scriptsize}
The entries of this matrix are better written using elementary symmetric polynomials, viz.
\begin{scriptsize}
\begin{equation*}
x_a x_b
\begin{vmatrix}
x_a^4 - x_a^3 e_1(45) + x_a^2 e_2(45)
& x_a^3 - x_a^2 e_1(45) + x_a e_2(45)
& x_a^2 - x_a e_1(45) + e_2(45)
& x_a - e_1(5) & 1 \\
x_b^4 - x_b^3 e_1(45) + x_b^2 e_2(45)
& x_b^3 - x_b^2 e_1(45) + x_b e_2(45)
& x_b^2 - x_b e_1(45) + e_2(45)
& x_b - e_1(5) & 1 \\
h_1(23) & 1 & 0 & 0 & 0 \\
0 & 0 & 0 & 1 & 0 \\
0 & 0  & 0  & 0 & 1
\end{vmatrix}.
\end{equation*}
\end{scriptsize}
Continuing to pivot 1 in row 3, column 2,
we multiply the second column by $- x_2 - x_3$ and add it to the  first column. 
The result is 
\begin{scriptsize}
\begin{equation*}
x_a x_b
\begin{vmatrix}
x_a^4 - x_a^3 e_1(2345) + x_a^2 e_2(2345) - x_a e_3(2345) + e_4(2345)
& x_a^3 - x_a^2 e_1(45) + x_a e_2(45)
& x_a^2 - x_a e_1(45) + e_2(45)
& x_a - e_1(5) & 1 \\
x_b^4 - x_b^3 e_1(2345) + x_b^2 e_2(2345) - x_b e_3(2345) + e_4(2345)
& x_b^3 - x_b^2 e_1(45) + x_b e_2(45)
& x_b^2 - x_b e_1(45) + e_2(45)
& x_b - e_1(5) & 1 \\
0 & 1 & 0 & 0 & 0 \\
0 & 0 & 0 & 1 & 0 \\
0 & 0  & 0  & 0 & 1
\end{vmatrix}.
\end{equation*}
\end{scriptsize}
which may be expressed as the smaller $2 \times 2$ determinant
\begin{scriptsize}
\begin{equation*}
x_a x_b
\begin{vmatrix}
x_a^4 - x_a^3 e_1(2345) + x_a^2 e_2(2345) - x_a e_3(2345) + e_4(2345)
& x_a - e_1(5)  \\
x_b^4 - x_b^3 e_1(2345) + x_b^2 e_2(2345) - x_b e_3(2345) + e_4(2345)
& x_b - e_1(5)
\end{vmatrix}.
\end{equation*}
\end{scriptsize}
The entries in this smaller determinant factor as
\begin{scriptsize}
\begin{equation*}
x_a x_b
\begin{vmatrix}
(x_a - x_2)(x_a - x_3)(x_a - x_4)(x_a - x_5)
& (x_a - x_5)  \\
(x_b - x_2)(x_b - x_3)(x_b - x_4)(x_b - x_5)
& (x_b - x_5)  \\
\end{vmatrix}.
\end{equation*}
\end{scriptsize}

For general $J = \{j_1 < \cdots < j_r \}$ and $K = \{k_1 < \cdots < k_r \}$, 
this procedure yields the formula
\begin{equation}
\label{p-small-determinant}
\FF_{J,K} = \pm \prod_{k \in K} x_k \cdot  \begin{vmatrix} \prod_{i > j_q} (x_{k_p} - x_i)   \end{vmatrix}_{1 \leq p, q \leq r}
\end{equation}
expressing $\FF_{J,K}$ as an $r \times r$ determinant times the variables indexed by $K$.
If $k_p > j_q$, the $(p,q)$-entry of the determinant in Equation~\eqref{p-small-determinant} vanishes.
If $J \not\leq_\Gale K$, this determinant has the block form
$\begin{vmatrix} * & * \\ {\bm 0} & * \end{vmatrix}$
where the southwest block of zeros intersects the main diagonal, so that $\FF_{J,K} = 0$.  
If $J = K$, the determinant in Equation~\eqref{p-small-determinant} is upper triangular, and the product of diagonal 
entries is as described in the statement of the lemma.
\end{proof}

\subsection{The colon ideal $(I_n : f_J)$ in $\CC[\xx_n]$}  Thanks to Lemma~\ref{f-product}, the differential operators
$\DD_J$ exhibit useful triangularity with respect to the Gale order on fermionic monomials.
In order to consider their fermionic leading term $\theta_J$, we will   study
the colon ideals
\begin{equation}
(I_n : f_J) := \{ g \in \CC[\xx_n] \,:\, g \cdot f_J \in I_n \} \subseteq \CC[\xx_n]
\end{equation}
where $I_n \subseteq \CC[\xx_n]$ is the classical coinvariant ideal.

It will turn out (Theorem~\ref{colon-ideal-identification}) that the ideal $(I_n : f_J)$ has two other equivalent definitions.
As a first step to proving this, we introduce the following bigraded subspace of $\Omega_n$.

\begin{defn}
\label{shprime-defn}
Let $SH'_n$ be the smallest  linear subspace of $\Omega_n$ which
\begin{itemize}
\item contains the superspace Vandermonde $\delta_n$, 
\item is closed under all bosonic partial derivatives $\partial_1, \dots, \partial_n$, and
\item is closed under the action of the higher Euler operators $d_i$ for $i \geq 1$.
\end{itemize}
\end{defn}

Swanson and Wallach showed \cite{SW2} that $SH'_n$ is annihilated by the supercoinvariant ideal $SI_n \subseteq \Omega_n$
under the $\odot$-action, so that $SH'_n \subseteq SH_n$ is a subset of the superharmonic space.
We will show (Theorem~\ref{superharmonic-space-characterization})
that in fact $SH'_n = SH_n$.  
For now, we can use $SH'_n$ and our triangularity results (Lemmas~\ref{q-polynomial-observations} and 
\ref{f-product})
to show that the polynomials $p_{J,1}, \dots, p_{J,n}$ from Section~\ref{Upper} lie in $(I_n : f_J)$.

\begin{lemma}
\label{p-polynomial-containment}
Let $J \subseteq [n]$.  For any $1 \leq i \leq n$ we have $p_{J,i} \in (I_n : f_J)$.
\end{lemma}

\begin{proof}
Let $q_{J,i} \in SI_n$ be the supercoinvariant ideal element associated to $p_{J,i}$. By Lemma~\ref{q-polynomial-observations} (3) 
we have
\begin{equation}
\label{q-triangularity}
q_{J,i} = p_{J,i} \cdot \theta_J + \sum_{J <_\Gale L} A_L \cdot \theta_L
\end{equation}
for some polynomials $A_L \in \CC[\xx_n]$.
On the other hand, Lemma~\ref{f-product} implies that 
\begin{equation}
\label{d-triangularity}
\DD_J(\delta_n) \doteq (f_J \odot \delta_n) \cdot \theta_J + \sum_{K <_\Gale J} B_K \cdot \theta_K
\end{equation}
for some $B_K \in \CC[\xx_n]$, where $\doteq$ denotes equality up to a nonzero scalar.
Since $\DD_J$ is a linear combination of $d_I$ operators with coefficients in $\partial_1, \dots, \partial_n$, we have
\begin{equation}
\DD_J(\delta_n) \in SH'_n \subseteq SH_n
\end{equation}
where the $\subseteq$ is justified by the work of Swanson and Wallach \cite{SW2}.
Since $SI_n$ annihilates $SH_n$ under the $\odot$-action and $q_{J,i} \in SI_n$, we have
\begin{equation}
\label{first-zero-superspace}
q_{J,i} \odot \DD_J(\delta_n) = 0.
\end{equation}
The triangularity relations \eqref{q-triangularity} and \eqref{d-triangularity} force
\begin{equation}
 (p_{J,i} \cdot f_J) \odot \delta_n = p_{J,i} \odot (f_J \odot \delta_n) = 0.
\end{equation}
Since $\ann_{\CC[\xx_n]}(\delta_n) = I_n$, this implies that $p_{J,i} \cdot f_J \in I_n$, or equivalently $p_{J,i} \in (I_n : f_J)$.
\end{proof}

The colon ideals $(I_n : f_J)$ are connected to a  class of permutations in $\symm_n$.
If $1 \leq j \leq n$, a permutation $w \in \symm_n$ is called {\em $j$-resentful} if
$w(j) = n$, or
the value $w(j)+1$ appears among $w(j+1), w(j+2), \dots, w(n)$.\footnote{We think of the one-line notation
$w = [w(1), \dots, w(n)]$ as recording the scores of $n$ musicians performing in a competition; after their performance, they sit down
and join the audience.
If the $j^{th}$ contestant  scores best  (i.e. $w(j) = n$)
or is beaten by 1 by an later contestant, this creates feelings of resentment (on behalf of the other contestants or the 
$j^{th}$ constant, respectively).}
The permutation $w$ is {\em $j$-Nietzschean} if it is not $j$-resentful.\footnote{The creator of The Superman 
should have some avatar in superspace.}

If $J \subseteq [n]$ is a subset, a permutation $w \in \symm_n$ is {\em $J$-Nietzschean} if it is $j$-Nietzschean for all $j \in J$.
We write
\begin{equation}
\NNNN_J := \{ w \in \symm_n \,:\, w \text{ is $J$-Nietzschean} \}
\end{equation}
for the set of all $J$-Nietszschean permutations in $\symm_n$.  
Nietzschean permutations are counted by a simple product formula.

\begin{proposition}
\label{superman-count}
Let $J \subseteq [n]$. The number of $J$-Nietzschean permutations in $\symm_n$ is given by
\begin{equation}
|\NNNN_J| = \prod_{i = 1}^n \stair(J)_i
\end{equation}
where $\stair(J) = (\stair(J)_1, \dots, \stair(J)_n)$ is the $J$-staircase.
\end{proposition}

\begin{proof}
We consider decomposing the one-line notation of permutations $w = [w(1), \dots, w(n)] \in \symm_n$ 
to the permutation $[1] \in \symm_1$ by iteratively removing the last letter $w(n)$ and `standardizing' to the unique
order-isomorphic permutation in $\symm_{n-1}$.  For example, the permutation
$[6,3,5,1,4,7,2] \in \symm_7$ decomposes as follows:
\begin{center}
$[6,3,5,1,4,7,2]$ \\
$[5,2,4,1,3,6]$ \\
$[5,2,4,1,3]$  \\
$[4,2,3,1]$ \\
$[3,1,2]$ \\
$[2,1]$ \\
$[1]$
\end{center}
Reversing this process, we can build up from $[1] \in \symm_1$ to a permutation in $\symm_n$ by appending a new letter to the 
end at each stage.
In order for the resulting permutation $w = [w(1), \dots , w(n)] \in \symm_n$ to be $J$-Nietzschean, suppose we have a permutation
$[v(1), \dots, v(k-1)] \in \symm_{k-1}$ at some intermediate stage and we want to build a permutation
in $\symm_{k}$. We may append any of the 
numbers in $\{1, \dots, k\}$  to  $[v(1), \dots, v(k-1)]$, except the following.
\begin{itemize}
\item If $k \in J$ is a Nietzschean position, we cannot append $k$, since this would ultimately 
force $w(k) = n$ or force an entry 1 larger than $w(k)$ to appear among $w(k+1), \dots, w(n)$, so that 
$w$ would be $k$-resentful.
\item  Whether or not $k$ is a Nietzschean position, we 
cannot append a value $v(j) + 1$ for any Nietzschean position $j \in J$ satisfying $j < k$, since this  
would ultimately force $w(j) + 1$ to appear among $w(j+1), \dots, w(n)$, so that $w$ would be $j$-resentful.
The value $v(j)$ at a Nietzschean position $j < k$ inductively satisfies $v(j) < k-1$.
\end{itemize}
In general, the conditions above imply that the number of choices to append to $[v(1), \dots, v(k-1)]$ is
\begin{equation}
k+1 - | \{ j \in J \,:\, j \leq k \}|,
\end{equation}
which yields the claimed product formula.
\end{proof}

We will see that $|\NNNN_J| = \dim \CC[\xx_n]/(I_n:f_J)$, so $J$-Nietzschean permutations enumerate bases of 
$\CC[\xx_n] / (I_n : f_J)$.
However, the connection between Nietzschean permutations and colon ideals goes  deeper than this.
To explain, we recall the powerful theory of orbit harmonics.

For any subset $Z \subseteq \CC^n$, let $\II(Z) \subseteq \CC[\xx_n]$ be the ideal of polynomials which vanish on $Z$:
\begin{equation}
\II(Z) := \{ f \in \CC[\xx_n] \,:\, f(\zz) = 0 \text{ for all $\zz \in Z$} \}.
\end{equation}
The quotient ring $\CC[Z] := \CC[\xx_n]/\II(Z)$ is the {\em coordinate ring} of $Z$ and has a natural identification
with the family of polynomial functions $Z \longrightarrow \CC$.  If we assume the locus $Z \subseteq \CC^n$ is 
finite (as we will from here on), by Lagrange interpolation {\bf any} function $Z \longrightarrow \CC$ is the restriction of a polynomial in 
$\CC[\xx_n]$, so we may identify $\CC[Z]$ with the vector space formal $\CC$-linear combinations of elements of $Z$.

The quotient ring $\CC[Z] = \CC[\xx_n]/\II(Z)$ is almost never graded, but there is a way to produce a graded quotient of 
$\CC[\xx_n]$ from $\II(Z)$.  For any nonzero polynomial $f \in \CC[\xx_n]$, let $\tau(f)$ be the highest degree homogeneous component
of $f$. That is, if $f = f_d + \cdots + f_1 + f_0$ where $f_i$ is homogeneous of degree $i$ and $f_d \neq 0$, we have $\tau(f) = f_d$.
We define a new ideal $\gr \, \II(Z) \subseteq \CC[\xx_n]$ by
\begin{equation}
\gr \, \II(Z) := (  \tau(f) \,:\, f \in \II(Z), \, f \neq 0 ) \subseteq \CC[\xx_n].
\end{equation}
The ideal $\gr \, \II(Z)$ is homogeneous by construction. We have an isomorphism of vector spaces
\begin{equation}
\label{orbit-harmonics-isomorphisms}
\CC[Z] = \CC[\xx_n]/\II(Z) \cong \CC[\xx_n]/ \gr \, \II(Z)
\end{equation}
where the latter quotient $\CC[\xx_n]/ \gr \, \II(Z)$ is a graded vector space.
The Hilbert series of $\CC[\xx_n]/ \gr \, \II(Z)$ may be regarded as a $q$-enumerator of $Z$ which depends in a subtle way on the 
embedding of $Z$ inside $\CC^n$.

As an example, if $Z = \symm_n$ is the set of points in $\CC^n$ of the form $[w(1), \dots, w(n)]$ for $w \in \symm_n$,
then $\gr \, \II(\symm_n) = I_n$ is the classical coinvariant ideal and the coinvariant ring 
$R_n = \CC[\xx_n]/I_n$ is obtained in this way.
The following result states that the colon ideals $(I_n : f_J)$ also arise via orbit harmonics.

\begin{theorem}
\label{colon-ideal-identification}
For any subset $J \subseteq [n]$, the following three ideals in $\CC[\xx_n]$ are equal.
\begin{enumerate}
\item  The colon ideal $(I_n : f_J)$.
\item  The ideal $(p_{J,1}, \dots, p_{J,n})$ generated by the homogeneous polynomials $p_{J,1}, \dots, p_{J,n} \in \CC[\xx_n]$.
\item  The homogeneous ideal $\gr \, \II(\NNNN_J)$ attached to the locus $\NNNN_J \subseteq \CC^n$ of
 $J$-Nietzschean permutations in $\symm_n$. Here we consider $\symm_n \subseteq \CC^n$ as the set of rearrangements 
 of the specific point $(1, 2, \dots, n) \in \CC^n$.
\end{enumerate}
If $\III_J \subseteq \CC[\xx_n]$ denotes this common ideal, the Hilbert series of $\CC[\xx_n] / \III_J$ is given by
\begin{equation}
\Hilb \left(  \CC[\xx_n] / \III_J ; q \right) = \prod_{i = 1}^n [\stair(J)_i]_q
\end{equation} 
where $\stair(J) = (\stair(J)_1, \dots, \stair(J)_n)$ is the $J$-staircase.
\end{theorem}

\begin{proof}
Suppose $1 \in J$. Lemma~\ref{f-containment-criterion} states that $f_J \in I_n$, so that  $(I_n : f_J) = \CC[\xx_n]$.
Furthermore, we have $p_{J,1} = \partial_1 h_1(x_1, \dots, x_n) = 1$, so that $(p_{J,1}, \dots, p_{J,n}) = \CC[\xx_n]$.
Finally, since every permutation $w \in \symm_n$ is 1-resentful, we have $\NNNN_J = \varnothing$ so that 
$\gr \, \II(\NNNN_J) = \CC[\xx_n]$.  Since $\stair(J)_1 = 0$, we are done in this case and assume that $1 \notin J$ going forward.

Lemma~\ref{p-polynomial-containment} yields the containment of ideals 
\begin{equation}
(p_{J,1}, \dots, p_{J,n}) \subseteq (I_n : f_J)
\end{equation}
so that $(2) \subseteq (1)$.  
We apply Lemma~\ref{colon-ideal-equality} with 
$\aaa = I_n, \aaa' = (p_{J,1}, \dots, p_{J,n}),$ and $f = f_J$. We check the conditions of this lemma.
\begin{itemize}
\item The ideal $I_n$ is generated by the regular sequence $e_1, \dots, e_n \in \CC[\xx_n]$.  The Artinian quotient 
$\CC[\xx_n]/(e_1, \dots, e_n)$ is a complete intersection, and hence Gorenstein. Artinian Gorenstein graded quotients of 
$\CC[\xx_n]$ are Poincar\'e duality algebras; see e.g. \cite[Prop. 2.1]{MW}. The socle
degree of $I_n$ is ${n \choose 2}$.\footnote{The ring $R_n = \CC[\xx_n]/I_n$ is also
 a Poincar\'e duality algebra because it presents
the cohomology of a compact smooth complex projective variety: the flag variety.} 
\item Since $1 \notin J$, Lemma~\ref{regular-sequence-lemma} implies that $p_{J,1}, \dots, p_{J,n}$ is a regular sequence, so that the 
quotient 
$\CC[\xx_n] / (p_{J,1}, \dots, p_{J,n})$ is also a Poincar\'e duality algebra. The socle degree of this algebra is 
$\deg p_{J,1} + \cdots + \deg p_{J,n} - n = \stair(J)_1 + \cdots + \stair(J)_n - n$.
\item Since $1 \notin J$, Lemma~\ref{f-containment-criterion} implies $f_J \notin I_n$. Furthermore, the polynomial $f_J$ has degree
$\deg f_J = \sum_{i = 1}^n (i - \stair(J)_i)$.
\end{itemize}
Since we have
\begin{equation}
\stair(J)_1 + \cdots + \stair(J)_n - n +  \sum_{i = 1}^n (i - \stair(J)_i) = {n \choose 2},
\end{equation}
we may apply Lemma~\ref{colon-ideal-equality} to conclude 
\begin{equation}
(p_{J,1}, \dots, p_{J,n}) = (I_n : f_J)
\end{equation}
so that (1) = (2). This also implies that the claimed Hilbert series formula holds for $\III_J = (1)$ or $(2)$.

For any radical ideals $\III, \JJJ \subseteq \CC[\xx_n]$, the colon ideal $(\III:\JJJ) = \{ f \in \CC[\xx_n] \,:\, f \cdot \JJJ \subseteq \III \}$ 
has the interpretation
\begin{equation}
\VV(\III:\JJJ) = \overline{ \VV(\III) - \VV(\JJJ)}
\end{equation}
in terms of varieties in $\CC^n$, where the bar stands for Zariski closure. 
If $\VV(\III)$ is a finite locus of points, the bar can be removed.

Write $\RRRR_J := \symm_n - \NNNN_J$ for the resentful complement of the $J$-Nietzschean permutations in $\symm_n$.
Recall that we take the specific embedding of $\symm_n \subset \CC^n$ by taking all rearrangements of the coordinates of
$(1, 2, \dots, n) \in \CC^n$.  This also embeds $\RRRR_J$ and $\NNNN_J$ inside $\CC^n$.

 The (inhomogeneous) polynomial
\begin{equation}
\tilde{f}_J := \prod_{j \in J} (x_j - n) \prod_{i > j} (x_j - x_i + 1)
\end{equation}
vanishes on $\RRRR_J$. 
In fact, we have 
\begin{equation}
\NNNN_J = \symm_n - \VV(\tilde{f}_J) = \VV(\tilde{I}_n) - \VV(\tilde{f}_J)
\end{equation}
where $\tilde{I}_n$ is the `deformed version' of the classical coinvariant ideal
\begin{equation}
\tilde{I}_n := \langle e_d(x_1, \dots, x_n) - e_d(1, \dots, n) \,:\, 1 \leq d \leq n \rangle.
\end{equation}
Since $\tilde{I}_n$ is radical and $\tilde{f}_J$ has no repeated factors, the Nullstellensatz implies
\begin{equation}
\II(\NNNN_J) = \II ( \VV(\tilde{I}_n) - \VV(\tilde{f}_J) ) = \II(\VV(\tilde{I}_n : \tilde{f}_J)) = \sqrt{ (\tilde{I}_n : \tilde{f}_J) } =
(\tilde{I}_n : \tilde{f}_J)
\end{equation}
where $\sqrt{\cdot}$ stands for the radical of an ideal. Taking associated graded ideals gives
\begin{equation}
\gr \, \II(\NNNN_J) = \gr \, (\tilde{I}_n : \tilde{f}_J) \subseteq ( \gr \, \tilde{I}_n : f_J ) = (I_n  : f_J)
\end{equation}
where the containment $\subseteq$  is justified by considering the leading term of a polynomial $\tilde{g} \in \CC[\xx_n]$
such that $\tilde{g} \cdot \tilde{f}_J \in \tilde{I}_n$.

For arbitrary ideals $\III$ and polynomials $f$,
 the containment $\gr \, (\III : f ) \subseteq (\gr \, \III : \tau(f) )$ can certainly be strict.
However, in our setting, Proposition~\ref{superman-count} and the fact that 
\begin{equation}
\dim \CC[\xx_n] / (I_n : f_J) = \prod_{i = 1}^n \stair(J)_i = |\NNNN_J|
\end{equation}
imply
\begin{equation}
|\NNNN_J| = \dim \CC[\xx_n] / \gr \, \II(\NNNN_J) \leq \dim  \CC[\xx_n]/(I_n:f_J) = |\NNNN_J|
\end{equation}
which forces $\gr \, \II(\NNNN_J) = (I_n : f_J)$ so that (1) = (3) and the theorem is proved.
\end{proof}

\section{Operator theorem and Hilbert series}
\label{Operator}

\subsection{Operator theorem}
We are ready to give our characterization of the harmonic space $SH_n = SI_n^{\perp} \subseteq \Omega_n$.
The following result was conjectured by Swanson and Wallach \cite{SW2}, and was previously conjectured by 
N. Bergeron, Li, Machacek, Sulzgruber, and Zabrocki (unpublished).

\begin{theorem}
\label{superharmonic-space-characterization}
{\em (Operator Theorem)}
The superharmonic space $SH_n \subseteq \Omega_n$ is generated as a $\CC[\xx_n]$-module under the $\odot$-action 
by $d_I(\delta_n)$ for subsets $I \subseteq [n-1]$.  In symbols, we have
\begin{equation}
\label{superharmonic-sum}
SH_n = \sum_{I \subseteq [n-1]} \CC[\xx_n] \odot d_I(\delta_n).
\end{equation}
\end{theorem}

The sum appearing in Theorem~\ref{superharmonic-space-characterization} is not direct. Since $d_i(\delta_n) = 0$ whenever
$i > n$ and  we have $d_i d_j = - d_j d_i$, Theorem~\ref{superharmonic-space-characterization} may be rephrased as follows.
\begin{quote}
{\em The superharmonic space $SH_n$ is the smallest linear subspace of $\Omega_n$ which
\begin{itemize}
\item contains the Vandermonde determinant $\delta_n$,
\item is closed under the differentiation operators $\partial_1, \dots, \partial_n$ acting on the $x$-variables, and
\item is closed under the higher derivative operators $d_i$ for $i \geq 1$.
\end{itemize}}
\end{quote}

\begin{proof} 
Observe that the sum on the RHS of Equation~\eqref{superharmonic-sum} is the space $SH'_n$ of 
Definition~\ref{shprime-defn}.
As explained after Definition~\ref{shprime-defn}, Swanson and Wallach proved \cite{SW2} that 
 $SH'_n \subseteq SH_n$.
Since $SR_n \cong SH_n$, Corollary~\ref{upper-bound-dimension} gives an upper bound on the dimension of $SH_n$.
In order to show that this containment is  an equality, we use the $\DD_J$ operators and the colon ideals
$(I_n : f_J)$ to show that the dimension of $SH'_n$ is sufficiently large.

Let $J \subseteq [n]$.
Applying the differential operator $\DD_J$ to $\delta_n$ yields an element $\DD_J(\delta_n) \in SH'_n$.
We use our lemmata to derive the following facts about the superspace element $\DD_J(\delta_n)$.

\begin{itemize}
\item  By Lemma~\ref{image-of-derivative} and the vanishing assertion of Lemma~\ref{f-product},
the coefficient of $\theta_K$ in $\DD_J(\delta_n)$ is zero unless $K \leq_{\Gale} J$.
\item By Lemma~\ref{image-of-derivative} and the product formula in Lemma~\ref{f-product}, the
coefficient of $\theta_J$ in $\DD_J(\delta_n)$ is 
$\pm \, f_J \odot \delta_n $.
\end{itemize}

For any element $f \in \Omega_n$, the annihilator 
\begin{equation}
\ann_{\CC[\xx_n]} f = \{ g \in \CC[\xx_n] \,:\, g \odot f = 0 \} \subseteq \CC[\xx_n]
\end{equation}
is an ideal in the polynomial ring $\CC[\xx_n]$. For any subset $J \subseteq [n]$, we calculate
\begin{equation}
\ann_{\CC[\xx_n]}( f_J \odot \delta_n) = ( \ann_{\CC[\xx_n]} \delta_n : f_J) = (I_n : f_J)
\end{equation}
where we used the fact that the annihilator of the Vandermonde $\delta_n$ is the classical coinvariant ideal $I_n$.
We claim that there exists a set $\BBB_n(J) \subseteq \CC[\xx_n]$ of homogeneous polynomials such that
\begin{itemize}
\item The set $\BBB_n(J)$ has degree generating function $\sum_{g \in \BBB(J)} q^{\deg(g)} = \prod_{i = 1}^n [\stair(J)_i]_q$ and
\item  the set 
$\{ g \odot (f_J \odot \delta_n) \,:\, g \in \BBB_n(J) \}$ of polynomials in $\CC[\xx_n]$ is linearly independent.
\end{itemize}
Indeed, Theorem~\ref{colon-ideal-identification} implies  that there exists 
a set $\BBB_n(J) \subseteq \CC[\xx_n]$ of homogeneous
 polynomials with the given degree generating function which descends to a linearly independent subset of 
 $\CC[\xx_n] / (I_n : f_J )$.
Since $\ann_{\CC[\xx_n]}(\delta_n) = I_n$,
for any such $\BBB_n(J)$ the set of polynomials 
$\{ g \odot (f_J \odot \delta_n) \,:\, g \in \BBB_n(J) \}$ will be linearly independent in $\CC[\xx_n]$.

We combine our observations to prove the theorem. Suppose that some linear combination
\begin{equation}
\sum_{J \subseteq [n]} \sum_{g_J \in \BBB_n(J)} c_{J,g_J} (g_J \cdot \theta_J) \in \Omega_n
\end{equation}
(where the $c_{J,g_J} \in \CC$ are scalars)
annihilates the space $SH'_n$ as a differential operator:
\begin{equation}
\left(\sum_{J \subseteq [n]} \sum_{g_J \in \BBB_n(J)} c_{J,g_J} (g_J \cdot \theta_J) \right) \odot SH'_n = 0.
\end{equation}
By fermionic homogeneity, we may as well assume that
\begin{quote}
$(\star)$ for all $J \subseteq [n]$ such that there is some $c_{J,g_J} \neq 0$, the set $J$ has a fixed size.
\end{quote}
In particular, for any $K \subseteq [n]$ we have 
\begin{equation}
\left(\sum_{J \subseteq [n]} \sum_{g_J \in \BBB_n(J)} c_{J,g_J} (g_J  \cdot \theta_J) \right) \odot \DD_K(\delta_n) = 0.
\end{equation}
Working towards a contradiction, assume that at least one of the scalars $c_{J,g_J} \in \CC$ is nonzero. 
Choose $J_0 \subseteq [n]$ minimal under the Gale order such that at least one $c_{J_0, g_{J_0}}$ is nonzero.
Letting $K = J_0$, we have
\begin{align}
0 &= \left(\sum_{J \subseteq [n]} \sum_{g_J \in \BBB_n(J)} c_{J,g_J} ( g_J \cdot \theta_J) \right) \odot \DD_{J_0}(\delta_n) \\
&\doteq 
 \left( \sum_{g_{J_0} \in \BBB_n(J_0)} c_{J_0,g_{J_0}} \cdot g_{J_0} \right)  \odot \text{(coefficient of $\theta_{J_0}$ in $\DD_{J_0}(\delta_n)$)} \\
&=  \sum_{g_{J_0} \in \BBB_n(J_0)} c_{J_0,g_{J_0}} \cdot g_{J_0} \odot \left[ \pm  f_{J_0} \odot \delta_n \right]
\end{align}
where the second equality follows from the homogeneity assumption $(\star)$ and our Gale minimality assumption and
$\doteq$ denotes equality up to a nonzero scalar.
The linear independence of the set $\{ g_{J_0} \odot  ( f_{J_0} \odot \delta_n ) \,:\, g_{J_0} \in \BBB_n(J_0) \}$ forces
$c_{J_0,g_{J_0}} = 0$ for all $g_{J_0} \in \BBB_n(J_0)$, which is a contradiction.

We have the chain of inequalities
\begin{equation}
\sum_J |\BBB_n(J) | \leq \dim SH'_n \leq \dim SH_n = \dim SR_n \leq \sum_J |\BBB_n(J)| 
\end{equation}
where the first inequality comes from the last paragraph, the second inequality follows because $SH'_n \subseteq SH_n$, 
the equality holds because $SH_n$ is the harmonic space to the quotient $SR_n$, and the last inequality holds because of 
Corollary~\ref{upper-bound-dimension}.
These are all equalities, forcing $SH_n = SH'_n$.
\end{proof}

\subsection{Hilbert series}
Our goal in this subsection is to calculate the Hilbert series of $SR_n$ and describe a method for producing bases of $SR_n$.
The key to our approach is the following general linear independence criterion.

\begin{lemma}
\label{general-linear-independence-lemma}
Suppose that for each $J \subseteq [n]$, we have a set $\CCC_n(J) \subseteq \CC[\xx_n]$ of homogeneous polynomials
such that $\CCC_n(J)$ descends to a linearly independent subset of $\CC[\xx_n] / (I_n : f_J)$. Then the set
$\CCC_n \subseteq \Omega_n$ given by
\begin{equation}
\CCC_n := \bigsqcup_{J \subseteq [n]} \CCC_n(J) \cdot \theta_J
\end{equation}
descends to a linearly independent subset of $SR_n$.
\end{lemma}

The proof of Lemma~\ref{general-linear-independence-lemma} is quite similar to the proof of
Theorem~\ref{superharmonic-space-characterization}.

\begin{proof}
If not, we could find scalars $c_{J, g_J} \in \CC$ not all zero so that
\begin{equation}
\sum_{J \subseteq [n]} \sum_{g_J \in \CCC_n(J)} c_{J, g_J} (g_J \cdot \theta_J) = 0 \quad \text{in $SR_n$}
\end{equation}
or equivalently
\begin{equation}
\left(   
\sum_{J \subseteq [n]} \sum_{g_J \in \CCC_n(J)} c_{J, g_J} (g_J \cdot \theta_J) 
\right) \odot SH_n = 0.
\end{equation}
If we choose $J_0 \subseteq [n]$ to be Gale-minimal such that $c_{J_0, g_{J_0}} \neq 0$ for some $g_{J_0} \in \CCC_n(J_0)$, 
the relation
\begin{equation}
\left(   
\sum_{J \subseteq [n]} \sum_{g_J \in \CCC_n(J)} c_{J, g_J} (g_J \cdot \theta_J) 
\right) \odot \DD_{J_0} (\delta_n)  = 0
\end{equation}
implies (just as in the proof of Theorem~\ref{superharmonic-space-characterization}) that
\begin{equation}
\sum_{g_{J_0} \in \CCC_n(J_0)} c_{J_0, g_{J_0}}  \cdot g_{J_0} \odot (f_{J_0} \odot \delta_n) = 0
\end{equation}
which contradicts the linear independence of $\CCC_n(J_0)$ in $\CC[\xx_n] / (I_n : f_{J_0})$.
\end{proof}

We have all the tools necessary to calculate the Hilbert series of $SR_n$.
This proves a conjecture \cite[Conj. 6.5]{SS} of Sagan and Swanson.

 \begin{theorem}
\label{hilbert-series}
The bigraded Hilbert series of $SR_n$ is
\begin{equation}
\Hilb(SR_n;q,z) =
 \sum_{k = 1}^n z^{n-k} \cdot [k]!_q \cdot \Stir_q(n,k).
\end{equation}
\end{theorem}

\begin{proof}
For all subsets  $J \subseteq [n]$, let $B_n(J) \subseteq \CC[\xx_n]$ be a family of homogeneous polynomials 
which descends to a basis of $\CC[\xx_n]/(I_n : f_J)$.
By Theorem~\ref{colon-ideal-identification}, the degree generating function for polynomials in $\BBB_n(J)$ is
\begin{equation}
\sum_{g_J \in \BBB_n(J)} q^{\deg(g_J)} = [\stair(J)_1]_q \cdots [\stair(J)_n]_q.
\end{equation}
Lemma~\ref{general-linear-independence-lemma} guarantees that 
$\BBB_n := \bigsqcup_{J \subseteq [n]} \BBB_n(J) \cdot \theta_J$ descends to a linearly independent subset of 
$SR_n$. On the other hand, Lemma~\ref{ss-dimension} shows that 
\begin{multline}
\Hilb(SR_n;q,z) \geq 
\sum_{J \subseteq [n]} \left( \sum_{g_J \in \BBB_n(J)} q^{\deg(g_J)}  \right) \cdot z^{|J|} \\ =
\sum_{k = 1}^n z^{n-k} \cdot [k]!_q \cdot \Stir_q(n,k) \geq \Hilb(SR_n;q,z)
\end{multline}
where the inequality is a consequence of Proposition~\ref{upper-bound-dimension}.
This forces the linearly independent subset $\BBB_n \subseteq SR_n$ to be a basis and the inequalities to be equalities.
\end{proof}

We present a recipe for building bases of $SR_n$ from bases of the various commutative quotients
$\CC[\xx_n] / (I_n : f_J)$.
We also show how bases of the quotients $\CC[\xx_n] / (I_n : f_J)$ induce bases of the superharmonic space
$SH_n$.  Since $\Omega_n = SH_n \oplus SI_n$, bases of $SH_n$ automatically descend to bases of 
$SR_n = \Omega_n / SI_n$.
Working in $SH_n$ can be useful for machine computations, since we do not need to consider cosets $f + SI_n \in SR_n$.

\begin{theorem}
\label{basis-recipe}
Suppose that, for every subset $J \subseteq [n]$, we have a set $\BBB_n(J) \subseteq \CC[\xx_n]$ of polynomials. Let
\begin{equation}
\BBB_n := \bigsqcup_{J \subseteq [n]} \BBB_n(J) \cdot \theta_J.
\end{equation}
The following are equivalent.
\begin{enumerate}
\item  For all $J \subseteq [n]$, the set $\BBB_n(J)$ descends to a basis of the quotient ring $\CC[\xx_n] / (I_n : f_J)$.
\item  We have a basis of the superharmonic space $SH_n$ given by
\begin{equation}
\bigsqcup_{J \subseteq [n]}
\left\{
 (b_J \cdot \theta_J \odot \DD_J(\delta_n)) \odot \DD_J(\delta_n) \,:\, b_J \in \BBB_n(J) 
\right\}.
\end{equation}
\end{enumerate}
Either of (1) or (2) implies the following.
\begin{enumerate}[start = 3]
\item  The set $\BBB_n$ descends to a basis of $SR_n$.
\end{enumerate}
\end{theorem}

\begin{proof}
The proof of Theorem~\ref{hilbert-series} shows that (1) implies (3), so it is enough to verify that (1) and (2) are equivalent.

We define a map $\Psi$ of vector spaces 
\begin{equation}
\Psi: \bigoplus_{J \subseteq [n]} \CC[\xx_n] / (I_n : f_J) \longrightarrow SH_n
\end{equation}
by the formula
\begin{equation}
\Psi: \left(  h_J  \right)_{J \subseteq [n]} \longmapsto \sum_{J \subseteq [n]} ( h_J \cdot \theta_J \odot \DD_J(\delta_n) ) \odot \DD_J(\delta_n).
\end{equation}
Since the coefficient of $\theta_J$ in $\DD_J(\delta_n)$ is $\pm (f_J \odot \delta_n)$, we have
\begin{equation}
\left [ (I_n : f_J) \cdot \theta_J \right] \odot \DD_J(\delta_n) = 0
\end{equation}
so that $\Psi$ is well-defined.

We claim that $\Psi$ is a bijection.
Theorems~\ref{colon-ideal-identification} and \ref{hilbert-series} 
imply that the domain and codomain of $\Psi$ have the same 
dimension, so it is enough to show that $\Psi$ is a surjection.
Indeed,
Lemma~\ref{f-product} implies $\DD_J(\delta_n) = (f_J \odot \delta_n) \cdot \theta_J + \Sigma$ where 
$\Sigma \in \bigoplus_{K <_\Gale J} \CC[\xx_n] \cdot \theta_K$.  As a consequence, we have
\begin{equation}
(\CC[\xx_n] \cdot \theta_J) \odot \DD_J(\delta_n) = \CC[\xx_n] \odot (f_J \odot \delta_n) 
\end{equation}
for each $J \subseteq [n]$.
On the other hand, Theorem~\ref{colon-ideal-identification} implies that $\CC[\xx_n] / (I_n : f_J)$ is Artinian Gorenstein with 
socle spanned by $f_J \odot \delta_n$.  It follows that 
\begin{equation}
\CC[\xx_n] \odot (f_J \odot \delta_n) = (I_n : f_J)^\perp
\end{equation}
as ideals in $\CC[\xx_n]$.
Working modulo the subspace $\bigoplus_{K <_\Gale J} \CC[\xx_n] \cdot \theta_K$ we have 
\begin{multline}
\left[(\CC[\xx_n] \cdot \theta_J) \odot \DD_J(\delta_n) \right] \odot \DD_J(\delta_n)  = (I_n : f_J)^\perp \odot \DD_J(\delta_n) \\  \equiv \CC[\xx_n] \odot \DD_J(\delta_n)
\mod \bigoplus_{K <_\Gale J} \CC[\xx_n] \cdot \theta_K.
\end{multline}
The surjectivity of $\Psi$ follows from induction on Gale order and Theorem~\ref{superharmonic-space-characterization}.
\end{proof}

\subsection{Superspace Artin monomials}
Theorem~\ref{basis-recipe} gives a recipe for finding bases $\BBB_n$ of $SR_n$ from bases $\BBB_n(J)$ of the commutative quotients
$\CC[\xx_n] / (I_n : f_J)$.
Although a generic set $\BBB_n(J) \subseteq \CC[\xx_n]$ of polynomials of the appropriate degrees will descend to a basis
of $\CC[\xx_n] / (I_n : f_J)$,
the complexity of the ideals $(I_n : f_J) \subseteq \CC[\xx_n]$ has so far obstructed progress on finding non-generic bases 
$\BBB_n(J)$ of $\CC[\xx_n] / (I_n : f_J)$.  We present a conjecture in this direction.

Define the set of {\em $J$-Artin monomials} by
\begin{equation}
\AAA_n(J) := \left\{ 
x_1^{a_1} \cdots x_n^{a_n} \,:\, a_i < \stair(J)_i
\right\}.
\end{equation}
That is, the set $\AAA_n(J)$ consists of monomials in $\CC[\xx_n]$ whose exponent sequences fit below the $J$-staircase.
We have $\AAA_n(J) = \varnothing$ whenever $1 \in J$. 
If $J = \varnothing$, then $\AAA_n(\varnothing) = \{ x_1^{a_1} \cdots x_n^{a_n} \,:\, a_i < i \}$ was proven by E. Artin \cite{Artin}
to descend to a basis of $R_n$.

\begin{conjecture}
\label{artin-conjecture}
For any subset $J \subseteq [n]$, the $J$-Artin monomials $\AAA_n(J)$ descend to a basis of $\CC[\xx_n] / (I_n : f_J)$.
\end{conjecture}

Artin's result \cite{Artin} proves Conjecture~\ref{artin-conjecture} when $J = \varnothing$. 
By Theorem~\ref{basis-recipe}, if Conjecture~\ref{artin-conjecture} is true, then
\begin{equation}
\AAA_n = \bigsqcup_{J \subseteq [n]} \AAA_n(J) \cdot \theta_J
\end{equation}
would descend to a basis for $SR_n$.  This would prove a conjecture \cite[Conj. 6.7]{SS} of Sagan and Swanson.
Thanks to Theorem~\ref{colon-ideal-identification}, for any given $J$ it would suffice to prove that 
$\AAA_n(J)$ is linearly independent in or spans $\CC[\xx_n] / (I_n : f_J)$.

We will give evidence for Conjecture~\ref{artin-conjecture} by showing that it holds when
$J = \{r+1, \dots, n-1, n\}$ is Gale-maximal.
This requires a preparatory lemma on certain ideals $\JJJ_{r,p,n} \subseteq \CC[\xx_n]$ generated by 
partial derivatives of $h$-polynomials.

\begin{lemma}
\label{almost-box-h-basis}
Let $r \geq 1$, let $1 \leq p \leq n+1$, and consider the ideal
\begin{equation}
\JJJ_{r,p,n} := \left(
\partial_1 h_r, \partial_2 h_r, \dots, \partial_{p-1} h_r, \partial_{p} h_{r+1}, \dots, \partial_{n-1} h_{r+1}, \partial_n h_{r+1}
\right) \subseteq \CC[\xx_n]
\end{equation}
generated by $n$ partial derivatives of homogeneous symmetric polynomials in the full variable set $\xx_n$. The set
of monomials
\begin{equation}
\MMM_{r,p,n} := \left\{
x_1^{b_1} \cdots x_n^{b_n} \,:\, b_i < r-1 \text{ for $i < p$ and } b_i < r \text{ for $i \geq p$}
\right\}
\end{equation}
descends to a basis for $\JJJ_{r,p,n}$.
\end{lemma}

Lemma~\ref{almost-box-h-basis} says that $\CC[\xx_n]/\JJJ_{r,p,n}$ shares the same monomial basis as the quotient by variable powers
$\CC[\xx_n]/(x_1^{r-1} ,\dots, x_{p-1}^{r-1}, x_p^r, \dots, x_n^r)$.
Since $\JJJ_{r,p,n}$ has inscrutable Gr\"obner theory, our proof of Lemma~\ref{almost-box-h-basis} relies on exact sequences.
Harada, Horiguchi, Murai, Precup, and Tymoczko used a similar style of argument to prove an analogous result
\cite[Thm. 7.1]{HHMPT} on an Artin-like basis for the cohomology rings of regular nilpotent Hessenberg varieties.

\begin{proof}
If $r = 1$ and $p > 1$ then $\partial_1 h_1 = \partial_1 (x_1 + \cdots + x_n) = 1 \in \JJJ_{r,p,n}$ so that $\JJJ_{r,p,n} = \CC[\xx_n]$
is the unit ideal. Since $\MMM_{1,p,n} = \varnothing$, the result is true in this case. We assume that $r > 1$ or $r = 1$ and $p = 1$ going
forward.

We leave it to the reader to verify the formula 
\begin{equation}
\label{quasi-euler}
x_1 \partial_1 h_r +  \cdots + x_{p-1} \partial_{p-1} h_r + \partial_{p} h_{r+1} + \cdots +  \partial_n h_{r+1} = C \cdot h_r
\end{equation}
where $C = r + n - p + 1$.  Since $1 \leq p \leq n+1$ and $r \geq 1$, we have $C > 0$ and Equation~\eqref{quasi-euler} implies that 
\begin{equation}
\label{solid-membership}
h_r \in \JJJ_{r,p,n}.
\end{equation}
In particular, if we let $S = [n] - \{p\}$ we have 
\begin{equation}
\label{clever-membership}
\partial_p h_{r+1} = \partial_p \left(  
x_p h_r + h_{r+1}(S)
\right) =
h_r + x_p \cdot \partial_p h_r \in \JJJ_{r,p,n} 
\end{equation}
so that $\JJJ_{r,p+1,n} \subseteq \JJJ_{r,p,n}$ and $\VV(\JJJ_{r,p,n}) \subseteq \VV(\JJJ_{r,p+1,n})$.
Swanson and Wallach \cite[Lem. 6.2]{SW2} showed that $\VV(\JJJ_{r,n+1,n}) = \{0\}$, so that $\VV(\JJJ_{r,p,n}) = \{0 \}$ (our assumptions 
on $r$ and $p$ guarantee that the generators of $\JJJ_{r,p,n}$ have positive degree).
Lemma~\ref{regular-sequence-criterion}
shows that the generating set of $\JJJ_{r,p,n}$ is a regular sequence, so that 
\begin{equation}
\label{J-ideal-hilbert-series}
\Hilb \left(
\CC[\xx_n] / \JJJ_{r,p,n}; q
\right) = [ r-1 ]_q^{p-1} \cdot [r]_q^{n-p+1}.
\end{equation}

The memberships \eqref{solid-membership} and \eqref{clever-membership} imply that $x_p \cdot \partial_p h_r \in \JJJ_{r,p,n}$, so that 
$x_p \cdot \JJJ_{r,p+1,n} \subseteq \JJJ_{r,p,n}$.  We therefore have an exact sequence
\begin{equation}
\label{preliminary-exact-sequence}
\frac{\CC[\xx_n]}{\JJJ_{r,p+1,n}}  \xrightarrow{ \, \, \times \, x_p \, \,}
\frac{\CC[\xx_n]}{\JJJ_{r,p,n}} \xrightarrow{ \, \, \mathrm{can.} \, \, }
\frac{\CC[\xx_n]}{\JJJ_{r,p,n} + (x_p)} \rightarrow 0
\end{equation}
where the first map is induced by multiplication by $x_p$ and the second map is the canonical projection.
The next step is to identify the target of the second map in this sequence in terms of a smaller variable set.

Let $\bar{\xx}_{n-1} = (x_1, \dots, x_{p-1}, x_{p+1}, \dots, x_n)$ be the variable set $\xx_n$ with $x_p$ removed.
Let 
\begin{equation}
\pi: \CC[\xx_n] \twoheadrightarrow \CC[\bar{\xx}_{n-1}]
\end{equation}
be the  surjection defined by $\pi(x_i) = x_i$ for $i \neq p$ and $\pi(x_p) = 0$.
Let $\bar{\JJJ}_{r,p,n-1} \subseteq \CC[\bar{\xx}_{n-1}]$ be the ideal with the same generating set as $\JJJ_{r,p,n-1}$, but 
in the variable set $\bar{\xx}_{n-1}$. Writing $S = [n] - \{p\}$, for any $d > 0$ and any $i \neq p$ we have the evaluation
\begin{multline}
\pi: \partial_i h_d \mapsto
\left[  
\partial_i h_d
\right]_{x_p \, \rightarrow \, 0} = 
\left[  
\partial_i  (x_p \cdot h_{d-1} + h_d(S))
\right]_{x_p \, \rightarrow \, 0}  \\ =
\left[  
x_p \cdot \partial_i  ( h_{d-1} + h_d(S))
\right]_{x_p \, \rightarrow \, 0} = \partial_i h_d(S)
\end{multline}
Furthermore, we have 
\begin{equation}
\pi: \partial_p h_d \mapsto
\left[  
\partial_p h_d
\right]_{x_p \, \rightarrow \, 0} = 
\left[  
\partial_p (x_p \cdot h_{d-1} + h_d(S))
\right]_{x_p \, \rightarrow \, 0} = h_{d-1}(S).
\end{equation} 
Comparing the generators of $\JJJ_{r,p,n}$ with those of $\bar{\JJJ}_{r,p,n-1}$ and using $h_r(S) \in \bar{\JJJ}_{r,p,n-1}$, 
we conclude that 
\begin{equation}
\pi \left( \JJJ_{r,p,n} + (x_p) \right) = \bar{\JJJ}_{r,p,n-1}
\end{equation}
so that the exact sequence \eqref{preliminary-exact-sequence}
induces a new exact sequence
\begin{equation}
\label{next-exact-sequence}
\frac{\CC[\xx_n]}{\JJJ_{r,p+1,n}}  \xrightarrow{ \, \, \times \, x_p \, \,}
\frac{\CC[\xx_n]}{\JJJ_{r,p,n}} \xrightarrow{ \, \, \psi \, \, }
\frac{\CC[\bar{\xx}_{n-1}]}{\bar{\JJJ}_{r,p,n-1}} \rightarrow 0
\end{equation}
where the surjection $\psi$ is induced by $\pi$.
The Hilbert series formula \eqref{J-ideal-hilbert-series} implies that the dimensions of the vector spaces on either side
of \eqref{next-exact-sequence} add to the dimension of the vector space in the middle, so the first map in
\eqref{next-exact-sequence} is injective and we have a short exact sequence
\begin{equation}
\label{final-J-exact-sequence}
0 \rightarrow
\frac{\CC[\xx_n]}{\JJJ_{r,p+1,n}}  \xrightarrow{ \, \, \times \, x_p \, \,}
\frac{\CC[\xx_n]}{\JJJ_{r,p,n}} \xrightarrow{ \, \, \psi \, \, }
\frac{\CC[\bar{\xx}_{n-1}]}{\bar{\JJJ}_{r,p,n-1}} \rightarrow 0.
\end{equation}
By  induction, we may assume that $\MMM_{r,p+1,n}$ descends to a basis of $\CC[\xx_n]/\JJJ_{r,p+1,n}$ and that 
\begin{equation}
\bar{\MMM}_{r,p,n-1} := \left\{
x_1^{b_1} \cdots x_{p-1}^{b_{p-1}} x_{p+1}^{b_{p+1}} \cdots x_n^{b_n} \,:\, b_i < r-1 \text{ for $i < p$ and }
b_i < r \text{ for $i > p$}
\right\}
\end{equation}
descends to a basis of $\CC[\bar{\xx}_{n-1}]/\bar{\JJJ}_{r,p,n-1}$.  The exactness of 
\eqref{final-J-exact-sequence} and the observation
\begin{equation}
\MMM_{r,p,n} = x_p \cdot \MMM_{r,p+1,n} \sqcup \bar{\MMM}_{r,p,n-1}
\end{equation}
guarantee that 
$\MMM_{r,p,n}$ 
descends to a basis for 
$\CC[\xx_n]/\JJJ_{r,p,n}$, which completes the proof. 
\end{proof}

\begin{proposition}
\label{artin-special-case}
Conjecture~\ref{artin-conjecture} is true when $J = \{r+1, \dots, n-1,n\}$ is a Gale-maximal subset of $[n]$.
\end{proposition}

\begin{proof}
By Theorem~\ref{colon-ideal-identification}, the generators of $(I_n : f_J) \subseteq \CC[\xx_n]$ are 
\begin{multline}
h_1(x_1, \dots, x_n), \, \, 
h_2(x_1, \dots, x_n), \quad \dots \quad
h_r(x_r, \dots, x_n), \\\
\partial_{r+1} h_{r+1}(x_{r+1}, \dots, x_n), \, \, 
\partial_{r+2} h_{r+1}(x_{r+1}, \dots, x_n), \quad \dots \quad
\partial_n h_{r+1}(x_{r+1}, \dots, x_n).
\end{multline}
Since $h_d(x_d, \dots, x_n) = x_d^d + \Sigma$ where $\Sigma$ is a linear combination of terms which are $> x_d^d$
in lexicographial order, we see that $\CC[\xx_n] / (I_n : f_J)$ is spanned by monomials of the form
$x_1^{b_1} \cdots x_n^{b_n}$ where $b_i < i$ for $i \leq r$.
The generators $\partial_i h_{r+1}(x_{r+1}, \dots, x_n)$ of $(I_n : f_J)$ and Lemma~\ref{almost-box-h-basis} (applied over the set
$\{x_{r+1}, \dots, x_n \}$ of variables indexed by $J$) implies that $\AAA_n(J)$ descends to a spanning set of 
$\CC[\xx_n] / (I_n : f_J)$.
This spanning set must be a basis by Theorem~\ref{colon-ideal-identification}.
\end{proof}

Given Proposition~\ref{artin-special-case}, a natural strategy for proving Conjecture~\ref{artin-conjecture}
would be to induct on the position of $J$ in Gale order. The base case of $J$  Gale-maximal is handled by Proposition~\ref{artin-special-case}.
If $i \notin J$ and $i+1 \in J$, we have $s_i \cdot J <_\Gale J$ where $s_i = (i,i+1)$ is the adjacent transposition in $\symm_n$.
Furthermore, the property $(\aaa : f g) = ( (\aaa : f) : g)$ of colon ideals gives rise to a natural injection
\begin{equation}
0 \rightarrow \frac{\CC[\xx_n]}{(I_n : f_{s_i \cdot J})} \xrightarrow{ \, \, \varphi \, \, }
\frac{\CC[\xx_n]}{(I_n : f_J)}
\end{equation}
where $\varphi(f) := (x_i - x_{i+1}) \times s_i \cdot f$ is defined by swapping the variables $x_i \leftrightarrow x_{i+1}$ and multiplying by
$x_i - x_{i+1}$.
Unfortunately, the map $\varphi$ does not relate to the structure of monomials in $\AAA_n(s_i \cdot J)$ and
$\AAA_n(J)$ in an obvious way; this has made Conjecture~\ref{artin-conjecture} resistant to inductive attack.

\section{Conclusion}
\label{Conclusion}

The most glaring open problem of our work is to enhance the Hilbert series result
of Theorem~\ref{hilbert-series} and prove the Fields Conjecture~\ref{fields-conjecture} on the bigraded $\symm_n$-structure of 
$SR_n$. One way to achieve this would be to show that the composite linear map
\begin{equation}
\varphi: 
\bigoplus_{k = 1}^n V_{n,k} \hookrightarrow \Omega_n \twoheadrightarrow SR_n
\end{equation}
is bijective, where $V_{n,k} \subseteq \Omega_n$ are the spaces constructed by the authors \cite{RWVan} and described in the introduction.
Thanks to Theorem~\ref{hilbert-series} and \cite{RWVan},
we know that the domain and target of $\varphi$ have the same vector space dimension,
so we are asking that $\varphi$ have a generic property.
Unfortunately, much like in the case of Conjecture~\ref{artin-conjecture}, proving that $\varphi$ satisfies this generic property has 
exhibited resistance to direct attack.

Various ideas in this paper have made appearances in the theory of Hessenberg varieties.
Lemma~\ref{colon-ideal-equality} on the realization of colon ideals $(\aaa : f)$ by complete intersections
was used by Abe, Horiguchi, Masuda, Murai, and Sato \cite{AHMMS} to relate the cohomology rings of Hessenberg varieties
to derivation modules of hyperplane arrangements associated to down-closed sets in positive root posets.
The polynomials $f_J \in \CC[\xx_n]$ appearing in this paper
 factor into products $\prod_{j \in J} f_{\{j\}}$ labeled by singletons.
In turn, the polynomials $f_{\{j\}}$ labeled by singletons resemble members of a family  $f_{j,i} \in \CC[\xx_n]$ 
of polynomials appearing in the work of Abe, Harada, Horiguchi, and Masuda \cite{AHHM}.
The polynomials $f_{j,i}$ were used to present the cohomology of regular nilpotent Hessenberg varieties using a GKM-style
excision which bears combinatorial resemblance to removing $J$-resentful permutations from $\symm_n$ to arrive at 
$J$-Nietzschean permutations.
An Artin-like basis of these cohomology rings was proven by 
Harada, Horiguchi, Murai, Precup, and Tymoczko \cite{HHMPT}; we use similar techniques
in the proof of Lemma~\ref{almost-box-h-basis} to show in Proposition~\ref{artin-special-case}
that the Artin monomials attached to terminal subsets 
$J = \{ r, r+1, \dots, n \} \subseteq [n]$ descend to a basis of the quotient rings $\CC[\xx_n] / (I_n : f_J)$.
Given these technical parallels, the authors suspect that there is a deeper connection between the supercoinvariant ring $SR_n$
and Hessenberg theory.
We present a conjecture in this direction as follows.

Recall that a finite-dimensional graded $\CC$-algebra $A = \bigoplus_{i = 0}^d A_i$ with $A_d \neq 0$ satisfies
Poincar\'e Duality if $A_d \cong \CC$ is 1-dimensional and if the multiplication $A_i \otimes A_{d-i} \rightarrow A_d \cong \CC$
is a perfect paring for all $0 \leq i \leq d$.
If $A$ satisfies Poincar\'e Duality, an element $\ell \in A_1$ of homogeneous degree 1 is a {\em Lefschetz element} if,
for all $i < d/2$, the map
\begin{equation}
\ell^{d - 2i} \times (-) : A_i \longrightarrow A_{d-i}
\end{equation}
of multiplication by $\ell^{d - 2i}$ is a bijection. If a Lefschetz element $\ell \in A_1$ exists, the algebra $A$ is said to satisfy the 
{\em Hard Lefschetz property}.

Algebras $A$ which satisfy PD and HL arise naturally in geometry. 
If $X$ is a smooth closed complex projective variety, its cohomology ring $A = H^{\bullet}(X)$ satisfies PD and HL 
(here we double the grading by setting $A_i := H^{2i}(X)$).
For example, the coinvariant ring $R_n = \CC[\xx_n]/I_n = H^{\bullet}(\mathrm{Fl}(n))$ satisfies PD and HL.
Maeno, Numata, and Wachi proved  \cite{MNW} that a linear form $\ell = c_1 x_1 + \cdots + c_n x_n$ is a Lefschetz element of 
$R_n$ if and only if the coefficients $c_1, \dots, c_n \in \CC$ are distinct.

Even if a variety $X$ is not smooth, its cohomology ring $H^{\bullet}(X)$ can still satisfy PD and HL.
Abe, Horiguchi, Masuda, Murai, and Sato  proved \cite[Thm. 12.1]{AHMMS} that $H^{\bullet}(X)$ satisfies PD and HL
when $X$ is a regular nilpotent Hessenberg variety, despite the fact that these varieties are usually singular.
Furthermore, a graded algebra $A = \bigoplus_{i = 0}^d A_i$ can still satisfy PD and HL, and so behave like the cohomology ring
of a hypothetical smooth compact variety $X$. 
As we have seen, the quotients $\CC[\xx_n] / (I_n : f_J)$ satisfy PD since they are complete intersections.
For the next conjecture, we adopt the convention that the zero ring $0 = H^{\bullet}(\varnothing)$ satisfies HL.

\begin{conjecture}
\label{hl-conjecture}
For any $J \subseteq [n]$, the quotient ring $\CC[\xx_n] / (I_n : f_J)$ satisfies the Hard Lefschetz property.
\end{conjecture}

Conjecture~\ref{hl-conjecture} has been tested for $n \leq 7$.  Computational data suggests that the linear forms
$\ell = c_1 x_1 + \cdots + c_n x_n$ continue to serve as Lefschetz elements, provided $c_1, \dots, c_n \in \CC$ are distinct.
We suspect that the Hodge-Riemann relations hold for $\CC[\xx_n] / (I_n : f_J)$, as well (see \cite[Sec. 12]{AHMMS}).

One of the most aesthetically pleasing aspects of $SR_n$ is its direct extension to general complex reflection groups.
An element $g \in GL_n(\CC)$ is a {\em pseudoreflection} if $g$ is conjugate to a diagonal matrix of the form
$\mathrm{diag}(\zeta,1,\dots,1)$ where $\zeta \in \CC^{\times}$ is a root-of-unity of finite order.
A finite subgroup $G \subseteq GL_n(\CC)$ is a {\em complex reflection group} if $G$ is generated by pseudoreflections.

The natural action of a complex reflection group $G \subseteq GL_n(\CC)$ on $\CC^n$ induces actions of $G$ on 
$\CC[\xx_n]$ and $\Omega_n$ by linear substitutions.
Chevalley proved \cite{Chevalley} that the invariant subring $\CC[\xx_n]^G$ admits a set $f_1, \dots, f_n$ 
of algebraically independent homogeneous generators of positive degrees, so that 
$\CC[\xx_n]^G = \CC[f_1, \dots, f_n]$ is itself a polynomial ring.
Although the $f_i$ are not unique, their degrees $d_1, \dots, d_n$ are uniquely determined by $G$.
Solomon \cite{Solomon} proved that the superspace invariants $(\Omega_n)^G$ are a free $\CC[\xx_n]^G$-module
and described a basis for this module as follows.

\begin{theorem}
\label{solomon-theorem}
{\em (Solomon \cite{Solomon})}
Let $f_1, \dots, f_n \in \CC[\xx_n]^{\symm_n}$ be any list of algebraically independent homogeneous generators
of $\CC[\xx_n]^{\symm_n}$.  The space $(\Omega_n)^{\symm_n}$  is a free module over 
$\CC[\xx_n]^{\symm_n}$ with basis
\begin{equation}
\{ d f_{i_1} \cdots d f_{i_r} \,:\, 0 \leq r \leq n, \, \, 1 \leq i_1 < \cdots < i_r \leq n \}.
\end{equation}
\end{theorem}

Solomon's Theorem~\ref{solomon-theorem} describes the space $(\Omega_n)^G$ of 
$G$-invariants as a $\CC[\xx_n]^G$-module. Any fundamental system of invariants $f_1, \dots, f_n \in \CC[\xx_n]^G$
gives rise to a generating set for the $G$-supercoinvariant ideal
$SI_G$ generated by $(\Omega_n)^G_+$.  We have $SI_G = (f_1, \dots, f_n, df_1, \dots, df_n)$ and may use this
presentation to study the quotient $SR_G := \Omega_n / SI_G$ as a bigraded $G$-module.

Solomon used Theorem~\ref{solomon-theorem} to give a uniform proof of the product formula
\begin{equation}
\sum_{g \in G} t^{\dim \mathrm{Fix}(g)} = (t + d_1 - 1) \cdots (t + d_n - 1)
\end{equation}
where $\mathrm{Fix}(g) = \{ v \in \CC^n \,:\, g \cdot v = v \}$ is the fixed subspace of $\CC^n$ attached to $g$.
In type A, this is equivalent to the factorization
\begin{equation}
\sum_{k = 0}^n c(n,k) \cdot t^k = t (t+1) \cdots (t + n - 1)
\end{equation}
where $c(n,k)$ is the Stirling number of the first kind counting permutations $w \in \symm_n$ with $k$ cycles.
On the other hand, the algebra of $SR_n = \Omega_n / SI_n$ is governed by ordered set partitions, which relate to 
Stirling numbers of the {\em second} kind.

Ordered set partitions of $[n]$ are in bijective correspondence with faces in the type A Coxeter complex.
All available
data in types BCD suggests that the fermionic degree $k$ piece of $SR_G := \Omega_n/SI_G$ has dimension equal 
to the number of codimension $k$
faces in the corresponding Coxeter complex (in type A this is a consequence of Theorem~\ref{hilbert-series}).
We also have agreement in type H$_3$.
However, in type F$_4$ these quantities disagree. The bigraded Hilbert series of $SR_{\mathrm{F}_4}$ is given by
\begin{scriptsize}
\begin{multline}
\Hilb(SR_{\mathrm{F}_4};q,z) = \\
\left(
\begin{array}{c}
 1 + 4q + 9 q^2 + 16 q^3 + 25 q^4  + 36 q^5 + 48 q^6 +  60 q^7 +  71 q^8 +  
80 q^9 +  87 q^{10} +  92 q^{11} +  94 q^{12} + \\  92 q^{13} + 87 q^{14} + 80 q^{15} + 71 q^{16} + 60 q^{17} +
 48 q^{18} +  36 q^{19} +  25q^{20} + 16 q^{21} + 9 q^{22} 4 + q^{23} + q^{24}
 \end{array}  \right) \cdot z^0 +  \\ 
 \left(
 \begin{array}{c}
 4  + 15q + 32q^2 + 55q^3 + 84q^4 +  118q^5  +152 q^6 + 182 q^7 + 204 q^8 + 215 q^9
 + 216 q^{10} + 207 q^{11} + \\  188 q^{12}  +  161 q^{13} +  132 q^{14} + 105 q^{15} + 80 q^{16} + 58 q^{17} +
  40 q^{18} + 26 q^{19} + 16 q^{20} + 9 q^{21} + 4 q^{22} + q^{23}
 \end{array}
 \right) \cdot z^1 + \\
 \left(
 \begin{array}{c}
 6 + 20 q + 39 q^2 + 64 q^3 + 95 q^4 +  128 q^5 + 154 q^6 + 168 q^7 +
  164 q^8 +  140 q^9 +  \\ 122 q^{10} + 100 q^{11} +   75 q^{12} + 52 q^{13} + 34 q^{14} + 20 q^{15} + 10 q^{16} + 4 q^{17}
  + q^{18}
 \end{array}
 \right) \cdot z^2  +
 \left(
 \begin{array}{c}
 4  +10 q + 16 q^2 + 25 q^3 + 36 q^4 + 43 q^5  + \\ 44 q^6 +  36  q^7 + 16 q^8 + 9 q^9 +  4 q^{10} + q^{11}
 \end{array}
 \right) \cdot z^3 + z^4
\end{multline}
\end{scriptsize}
and this expression has $q \rightarrow 1$ specialization
\begin{equation}
\Hilb(SR_{\mathrm{F}_4};1,z) =1152 \cdot z^0 + 2304 \cdot z^1 + 1396 \cdot z^2 + 244 \cdot z^3 + z^4.
\end{equation}
This coefficient sequence is almost the same as the reversed $f$-vector 
$(1152, 2304, 1392, 240, 1)$ of the type F$_4$ Coxeter complex, but the coefficients of $z^2$ and $z^3$ are too large by 4.  
Finding a precise invariant-theoretic description of the Hilbert series of $SR_G$ would likely be very interesting.

\section{Acknowledgements}
\label{Acknowledgements}

The authors are grateful to Fran\c cois Bergeron,
Nantel Bergeron, Darij Grinberg, Jim Haglund, Dani\"el Kroes, Yasuhide Numata, Vic Reiner,
 Bruce Sagan,
Josh Swanson, Nolan Wallach, Tianyi Yu, and Mike Zabrocki for many helpful conversations.
B. Rhoades was partially supported by NSF Grant DMS-1953781.

\end{document}